\documentclass[12pt]{amsart}

\usepackage{fullpage}
\usepackage{standalone}

\usepackage[utf8]{inputenc}
\usepackage[T1]{fontenc}
\usepackage{lmodern}

\usepackage{graphicx}
\usepackage{tikz}

\newtheorem{theorem}{Theorem}[section]
\newtheorem{proposition}[theorem]{Proposition}
\newtheorem{lemma}[theorem]{Lemma}
\newtheorem{corollary}[theorem]{Corollary}

\newtheorem{remark}[theorem]{Remark}
\newtheorem{question}[theorem]{Question}
\newtheorem{conjecture}[theorem]{Conjecture}

\numberwithin{equation}{section}
\newcommand\HH {{\mathbb H}}
\newcommand\KK {{\mathbb K}}
\newcommand\NN {{\mathbb N}}

\newcommand\QQ {{\mathbb Q}}
\newcommand\RR {{\mathbb R}}
\newcommand\TT {{\mathbb T}}
\newcommand\ZZ {{\mathbb Z}}

\newcommand\sltwoz{{\rm SL(2,\ZZ)}}
\newcommand\sltwor{{\rm SL(2,\RR)}}

\newcommand\psltwoz{{\rm PSL(2,\ZZ)}}

\newcommand\veech{{\rm SL}}

\newcommand\sys{{\rm Sys }}
\newcommand\area{{\rm Area }}

\newcommand\re{{\rm Re }}
\newcommand\im{{\rm Im }}
\newcommand\hol {{\rm Hol}}
\newcommand\support{{\rm Supp }}

\newcommand\cA{{\mathcal{A}  }}

\newcommand\cF{{\mathcal{F}  }}
\newcommand\cG{{\mathcal{G}  }}
\newcommand\cH{{\mathcal{H}  }}
\newcommand\cI{{\mathcal{I}  }}

\newcommand\cK{{\mathcal{K}  }}
\newcommand\cL{{\mathcal{L}  }}
\newcommand\cM{{\mathcal{M}  }}

\newcommand\cO{{\mathcal{O}  }}

\newcommand\cS{{\mathcal{S}  }}

\begin{document}

\title[Lagrange Spectrum Orbit B7]
{The Lagrange spectrum of some square-tiled surfaces}

\author{Pascal Hubert, Samuel Leli\`evre, Luca Marchese, Corinna Ulcigrai}

\address{I2M, Centre de Math\'ematiques et Informatique (CMI),
Universit\'e Aix-Marseille, 39 rue Joliot Curie, 13453 Marseille Cedex 13, France.}

\email{pascal.hubert@univ-amu.fr}

\address{Laboratoire de math\'ematique d'Orsay,
UMR 8628 CNRS Université Paris-Sud,
Bâtiment 425, 91405 Orsay cedex, France}

\email{samuel.lelievre@math.u-psud.fr}

\address{Universit\'e Paris 13, Sorbonne Paris Cit\'e,
LAGA, UMR 7539, 99 Avenue Jean-Baptiste Cl\'ement, 93430 Villetaneuse, France.}

\email{marchese@math.univ-paris13.fr}

\address{School of Mathematics, University of Bristol, University Walk, Clifton, Bristol, BS8 1TW, United Kingdom.}

\email{corinna.ulcigrai@bristol.ac.uk}

\begin{abstract}
Lagrange spectra have been defined for closed submanifolds of the moduli space of translation surfaces which are invariant under the action of $\sltwor$. We consider the closed orbit generated by a specific covering of degree 7 of the standard torus, which is an element of the stratum $\cH(2)$. We give an explicit formula for the values in the spectrum, in terms of a cocycle over the classical continued fraction. Differently from the classical case of the modular surface, where the lowest part of the Lagrange spectrum is discrete, we find an isolated minimum, and a set with a rich structure right above it.
\end{abstract}

\maketitle

\section{Introduction}

The classical \emph{Lagrange spectrum} $\cL$ is a famous and well studied subset of the real line which admits both number theoretical and dynamical interpreations. It  can be defined as the set  of values of the function $L:\RR\to\RR \cup \{ + \infty\}$ given by
$$
L(\alpha):=
\limsup_{q,p\to\infty}\frac{1}{q|q\alpha -p|}
$$
for $\alpha \in \RR$. One has that $L(\alpha)=L<+\infty$ if and only if $\alpha$ \emph{badly approximable}, or, more precisely if and only if 
for any $C>L$, we have $|\alpha-p/q|>(Cq^2)^{-1}$ for all $p$ and $q$ big enough and, moreover, $L$ is minimal with respect to this property. One can show that $\cL$ can also be described as a \emph{penetration spectrum} for the geodesic flow on the (unit tangent bundle of the) modular surface $\HH/\sltwoz$ in the following way. Let $\textrm{ height }(\cdot)$ be the hyperbolic height function on the modular surface. If $(\gamma_t)_{t\in\RR}$ is any hyperbolic geodesic on $\HH/\sltwoz$ which has $\alpha \in \RR$ as forward endpoint and which is contained in a compact, the value $L(\alpha)$ is related to the asymptotic depth of penetration of the geodesic $\gamma_t$ into the cusp of the modular surface, given by the quantity:
$$
\limsup_{t\to+\infty} \textrm{ height }(\gamma_t).
$$

The structure of $\cL$ has been studied for more than a century. Among the wide literature, we mention the very good introduction in \cite{cusicflahive}, where it is proved that $\cL$ is a closed subset of the real line and that the values $L(\alpha)$ for $\alpha$ quadratic irrational form a dense subset. M. Hall in 1947 (see \cite{hall}) established that  $\cL(\cM)$ contains a positive half-line, also known as \emph{Hall's ray}. 
 The lower part of $\cL$, that is the window $[\sqrt{5},3]\cap\cL$, is a discrete sequence of quadratic irrational numbers cumulating to $3$, as proved by Markoff in 1879. The minimum of $\cL$, that is $\sqrt{5}$, is known as \emph{Hurwitz's constant}. The geometric picture corresponding to the discrete part of $\cL$ was described in \cite{series}. A fine result on the Hausdorff dimension of $\cL$ in the region between the discrete part and the Hall ray is proved more recently by \cite{gugu}. 

A great variety of generalizations of Lagrange spectra has been defined and studied, mostly related to geodesic flows in negative curvature, in particular by Paulin-Parkkonen (see \cite{paulinparkkonnen} and references therein). We do no attempt to resume the literature on the subject and we refer to \cite{cusicflahive}, \cite{paulinparkkonnen} and \cite{hubertmarcheseulcigrai} for a complete bibliography.  The study of generalizations of the Lagrange spectrum has seen a very  recent surge of interest, see for example the recent work by Ibarra and Moreira on Lagrange spectra for the geodesic flow on surfaces with variable negative curvature \cite{IM} and in particular the various recent works \cite{hubertmarcheseulcigrai, artigianimarcheseulcigrai, boshernitzandelecroix} on Lagrange spectra of translation surfaces, which are also the object of study of the present paper. 

A generalization of Lagrange spectra in the context of \emph{translation surfaces} was recently introduced by Hubert, Marchese and Ulcigrai  in~\cite{hubertmarcheseulcigrai}. 
\emph{Translation surfaces} are surfaces obtained by glueing a finite set of polygons in the plane, by identifying pairs of isometric parallel sides by translations (see the next subsection definitions), so that  one gets a surface which carries a flat metric with conical singularities. The \emph{moduli space} of translation surfaces consists of translation surfaces up to the equivalence relation obtained by cutting and pasting by parallel translations.  A natural \emph{action of  $SL(2,\mathbb{R})$} (the group of $2$ by $2$ matrices with real entries and determinant one) on the moduli space of translation surfaces is induced by the linear action of matrices on polygons in the plane (see the next subsection).

More precisely, in \cite{hubertmarcheseulcigrai}, Lagrange spectra are associated to any \emph{locus} of translation surfaces in their moduli space which is closed and invariant under the $SL(2,\mathbb{R})$ action. Examples of $SL(2,\mathbb{R})$-invariant loci are \emph{strata}, i.e. loci formed by all translation surfaces with a prescribed type of conical singularities and \emph{Teichmueller curves}, i.e. (unit tangent bundles of) hyperbolic surfaces which are  closed $SL(2,\mathbb{R})$ orbits. The translation surfaces whose $\sltwor$ orbit is a \emph{Teichmueller curve} are known as \emph{Veech} (or \emph{lattice}) \emph{surfaces} and are rich in affine symmetries. The simplest example of a Veech surface is the square torus (whose $\sltwor$ orbit is the modular surface, for which the Lagrange spectrum defined in \cite{hubertmarcheseulcigrai} reduces to the classical case) and their covers, \emph{square tiled surfaces} (see below for the definition). 

It was shown that Lagrange spectra of closed invariant loci of translation surfaces share several of the same qualitative properties with $\cL$. For example, it is proved already  in \cite{hubertmarcheseulcigrai} that each of these generalized Lagrange spectra is closed and equal to the closure of a set of values which generalize quadratic irrationals (see after Definition \ref{defLspectrum}).   
In \cite{artigianimarcheseulcigrai} it is proved that $\cL(\cM)$ contains a positive half-line, also called \emph{Hall's ray}, for all those $\cM$ containing a \emph{Veech surface} $S$ (see also Theorem 1.6 of \cite{hubertmarcheseulcigrai}, where the same result was proved for loci which contain a square-tiled surface). Recently, M. Boshernitzan and V. Delecroix computed the minimum of the Lagrange spectrum of a connected component of a stratum of translation surfaces (whose value is given in the next subsection after Definition \ref{defLspectrum}) and showed that it is an isolated point (see \cite{boshernitzandelecroix}).  

The aim of this paper is to study an explicit simple case of Lagrange spectrum in the context of translation surfaces, for which we can  provide a very fine analysis  which displays interesting  phenomena which makes it different than the classical and general case (see below).  We consider \emph{square-tiled surfaces},  that is translation surfaces tiled by  copies of the square $[0,1]^2$, which  
as reminded earlier 
have a particularly simple  $\sltwor$ orbit closures in strata of translation surfaces: 
any orbit $\sltwor\cdot S$ of a square tiled surface $S$ is closed in its stratum. One can associate to the orbit closure a Lagrange spectrum (see the Definition \ref{defLvalues} and \ref{defLspectrum} below), which will denote by $\cL(S)$ (where $\cL(S)=\cL(S')$ if $S$ and $S'$ have the same orbit closure).  
  It was  shown in  \cite{hubertmarcheseulcigrai} that one can study the Lagrange spectrum of  square-tiled surfaces  by using renormalization and continued fractions. In particular, a formula for \emph{high} values in the Lagrange spectrum using continued fraction was proved (see Lemma \ref{Lem:Lem5.10HubertMarcheseUlcigrai} in \cite{hubertmarcheseulcigrai}). 
   

One can see that if $S$ 
is any square-tiled surface whose 
     number of squares is at most $5$, then $\cL(S)$ is an affine copy of the classical Lagrange spectrum $\cL(\TT^2)$.      
   This paper is devoted to the study of the Lagrange spectrum $\cL(S)$ arising for a square tiled surace $S$ consisting of $7$ squares (in the so-called \emph{Orbit B7} in the stratum $\cH(2)$, see below).   As explained in the next subsection (see also Remark \ref{RemWhyWeChoseOrbitB}), the $SL(2,R)$ orbit closure of this surface is one of the simplest orbit closures whose study is both non-trivial and accessible by renormalization techniques.
In particular,  one can show that the formula using continued fraction and renormalization allows to compute \emph{all} values in  the Lagrange spectrum of this translation surface (see Theorem \ref{thm:renormalizedformula} and Remark \ref{RemWhyWeChoseOrbitB}), while it is not clear if this is the case for an arbitrary number of squares.
 
 Our main result (Theorem \ref{TheoremMainTheorem} in the next section)  describes the initial structure of the Lagrange spectrum $\cL(S)$ of this square-tiled surface. We compute in particular the smallest value $\phi_1$ of $\cL(S)$ and prove that it is isolated, while we show that the second smallest value  $\phi_2$ is not isolated. In particular, this shows that  $\cL(S)$
 provides an example of a Lagrange spectrum whose lower part is richer than in the classical case and for which the minimum $\phi_1$ is different than the minimum of the corresponding connected component  computed by M. Boshernitzan and V. Delecroix in \cite{boshernitzandelecroix}.  Moreover, our analysis provide further information on the intricate structure of the bottom part of $\cL(S)$, i.e. we can describe two countable families of accumulation points above $\phi_2$ which seems to indicate that any interval with left endpoint in $\phi_2$ intersects $\cL(S)$ in a Cantor set with positive Hausdorff dimension bounded from below (see the open questions \ref{Q1} and \ref{Q2} below). The precise statement of our main results is given in Theorem \ref{TheoremMainTheorem} below. 

Our methods are a generalization of the methods used by Cusick and Flahive to analyze the discrete part of the classical Lagrange spectrum (see  \S~1 of \cite{cusicflahive}. We crucially exploit the renormalization formula for values in the spectrum 
(see Theorem \ref{thm:renormalizedformula} and Corollary \ref{cor:renormalizedformulafororbitB7}). As in the classical case, the formula is based on continued fractions, but in our setup also involves renormalized values of the \emph{multiplicity}, i.e. a geometric quantity associated to the $SL(2, \mathbb{Z})$ orbit of the square-tiled surface. By analyzing the structure of this orbit, which can be encoded in a graph, we show that to study the bottom of the spectrum it is sufficient to focus only on a smaller and simpler subgraph. Essentially, this allows to reduce the study of the Lagrange spectrum to the analysis of a limsup function over a subshift of finite type (see Theorem \ref{TheoremReductionToSubshift}). We believe that this type of framework has the potential to be applied in much greater generality in the study of other Lagrange spectra and that our present analysis, in addition to its intrinsic interest, can also provide  a paradigm for future works.  


\subsection*{Definitions and main results}

A \emph{translation surface} is a genus $g$ closed surface $S$ with a flat metric and a finite set $\Sigma$ of conical singularities $p_1,\dots,p_r$, the angle at each $p_i$ being an integer multiple of $2\pi$. An equivalent definition of translation surface $S$ is the datum $(M,w)$, where $M$ is a compact Riemann surface and $w$ is an holomorphic 1-form on $M$ having a zero at each $p_i$. The relation $k_1+\dots+k_r=2g-2$ holds, where $k_1,\dots,k_r$ are the orders of the zeroes of $w$. The 1-form $w$ induces a non-zero area form $(i/2)w\wedge\bar{w}$ on $S$ and we fix a normalization on translation surfaces requiring $\area(S)=1$. A stratum $\cH=\cH(k_1,\dots,k_r)$ is the set of translation surfaces $S$ whose corresponding holomorphic 1-form $w$ has $r$ zeros with orders $k_1,\dots,k_r$, where $k_1+\dots+k_r=2g-2$. Any stratum admits an action of $\sltwor$, indeed for a translation surface $S=(M,w)$ and an element $G\in\sltwor$ a new translation surface
$
G\cdot S=(G_\ast M,G_\ast w)
$
is defined, where the 1-form $G_\ast w$ is the composition of $w$ with $G$ and $G_\ast M$ is the complex atlas for which $G_\ast w$ is holomorphic. In the following we mostly consider the action of the geodesic flow $g_t$ and the group of rotations $r_\theta$, defined respectively for $t\in\RR$ and $-\pi\leq\theta<\pi$ by
$$
g_t:=
\begin{pmatrix}
e^{t} & 0 \\
0 & e^{-t}
\end{pmatrix}
\textrm{ and }
r_\theta:=
\begin{pmatrix}
\cos\theta & -\sin\theta \\
\sin\theta & \cos\theta
\end{pmatrix}.
$$
For a general overview on translation surfaces we recommend the surveys \cite{fornimatheus} and \cite{zorich}.

\medskip

A \emph{saddle connection} for the translation surface $S$ is a geodesic $\gamma$ for the flat metric connecting two conical singularities and not containing points of $\Sigma$ in its interior. The set $\hol(S)$ of \emph{periods} of $X$ is the set of complex numbers $v:=\int_\gamma w$, where $\gamma$ is a saddle connection for $S$ and $w$ is the holomorphic 1-form. The systole function $\sys:\cH\to\RR_+$ is defined by
$$
\sys(S):=\min_{v\in\hol(S)}|v|.
$$
Strata are non compact. An exhaustion by compact subsets is the family $(\cK)_{c>0}$, where for any $c>0$ we define $\cK_c$ as the set of those $S$ such that $\sys(S)\geq c$. Therefore the positive $g_t$-orbit of a translation surface $S$ stays in a compact if $\sys(g_t\cdot S)>c$ for some constant $c>0$ and any $t>0$. It is known from Vorobets, and reproved in Proposition 1.1 of \cite{hubertmarcheseulcigrai}, that
\begin{equation}\label{defLvalues}
L(S):=
\limsup_{|\im(v)|\to\infty}
\frac{1}{|\re(v)|\cdot|\im(v)|}=
\limsup_{t\to+\infty}
\frac{2}{\big(\sys(g_t\cdot S)\big)^2},
\end{equation}
where the $\limsup$ in the middle term is taken over all periods $v\in\hol(S)$. Therefore the quantity $L(S)$ defined above gives a measure of the size of the asymptotical maximal excursion of the positive $g_t$-orbit of $S$. Let $\cM$ be an orbit closure for the action of $\sltwor$ in some stratum. The associated \emph{Lagrange spectrum} is the set of values
\begin{equation}\label{defLspectrum}
\cL(\cM):=\{L(S)\textrm{ ; }S\in\cM\}.
\end{equation}
According to Theorem 1.5 of \cite{hubertmarcheseulcigrai}, any $\cL(\cM)$ is a closed subset of the real line, equal to the closure of the set of values $L(S)$ for those $S\in\cM$ whose geodesic $g_t\cdot S$ is periodic.  
M. Boshernitzan and V. Delecroix proved in \cite{boshernitzandelecroix} that if $\cM$ is a connected component of any stratum $\cH(k_1,\dots,k_r)$, then the minimum of $\cL(\cM)$ is $(k_1+\dots+k_r+r)\sqrt{5}$ and is isolated.

\medskip

In this paper we consider \emph{square-tiled surfaces} $S$, also said \emph{origamis}, that is translation surfaces $S$ tiled by copies of the square $[0,1]^2$. Equivalently, $S$ is square-tiled if there exists a ramified covering $\rho:S\to \RR^2/\ZZ^2$, unramified outside $0\in \RR^2/\ZZ^2$ and such that $\rho^\ast(dz)$ is the holomorphic 1-form of $S$. Square tiled surfaces are examples of Veech surfaces and any $\sltwor$ orbit  $\sltwor\cdot S$ of a square-tiled surface $S$ is closed in its stratum. For simplicity we denote its Lagrange spectrum by $\cL(S)$ instead of $\cL(\sltwor\cdot S)$, where of course $\cL(S)=\cL(S')$ if $S$ and $S'$ have the same orbit closure. Observe that it is implicit in our definition that the vertical and the horizontal directions on $S$ are rational directions. Consider any other direction and let $\theta$ be the angle it forms with the vertical, then set $\alpha:=\tan\theta$. At page 187 of \cite{hubertmarcheseulcigrai} it is explained that $\cL(S)$ is parametrized by the function
$
L(S,\cdot):\RR\to\cL(S)
$
defined by
$$
L(S,\alpha):=
L(r_{\arctan(\alpha)}\cdot S).
$$

It is an easy exercise to prove that for $S=\TT^2$ we have
$$
L(\TT^2,\alpha)=
\limsup_{q,p\to\infty}
\frac{1}{q\cdot|q\alpha-p|},
$$
which is the function giving rise to the classical Lagrange spectrum $\cL$, that is the set of values $L(\TT^2,\alpha)$ for $\alpha$ badly approximable. 

\medskip
 
This paper is devoted to the study of the Lagrange spectrum $\cL(S)$ arising for an origami $S$ in the so-called \emph{Orbit B7} in the stratum $\cH(2)$. Observe first that the action of $\sltwor$ induces an action of $\sltwoz$, and the latter preserves the set of origamis with a given number $N$ of squares. For the stratum $\cH(2)$ the orbits of such action have been classified in \cite{hubertlelievre} and \cite{mcmullen}, proving that there are always at most two orbits. The orbit B7 is described explicitly in \S~\ref{Sec:DescriptionOfOrbitB7} and essential details on the classification for $\cH(2)$ are resumed in \S~\ref{sec:action-sltwoz-primitive-origamis}. The interest in studying the orbit B7 can be seen from Figure \ref{fig:surfaces}: it contains origamis $X_j$ such that any horizontal saddle connection of $X_j$ is winded at least twice on the horizontal closed geodesic of $\TT^2$ under the covering map $\rho:X_j\to\TT^2$. One can see that if $\cO$ is an orbit where such property is not satisfied, like for example when the number of squares is at most $5$, then $\cL(\cO)$ is an affine copy of the classical Lagrange spectrum $\cL(\TT^2)$. On the other hand, any $X_j$ in the orbit B7 always has an horizontal saddle connection that is winded on the horizontal closed geodesic of $\TT^2$ at most twice. When this second property holds one can apply a nice renormalization formula to compute the Lagrange spectrum (see Theorem \ref{thm:renormalizedformula} and Remark \ref{RemWhyWeChoseOrbitB}), whereas we do not know if the formula holds for arbitrary number of squares, where in general the second property does not hold. The two properties mentioned above make the orbit B7 one of the simplest orbit closures whose study is both non-trivial and accessible by renormalization techniques. In the notation of \S~\ref{sec:action-sltwoz-primitive-origamis}, the other orbits with at most $7$ squares and satisfying the same two properties are the orbit A7, which has $54$ elements, and the orbit with 6 squares in $\cH(2)$, which has $36$ elements. Since the orbit B7 has $36$ elements we considered it as the easiest case to start from. 


\medskip
Our main result describes the bottom of the Lagrange spectrum $\cL(S)$. In order to state it, let us recall some continued fractions notation. 
Any irrational real number $\alpha$ admits an unique continued fraction expansion
$
\alpha=a_0+[a_1,a_2,\dots]
$,
where $a_0$ is the integer part of $\alpha$ and where
$$
[a_1,a_2,\dots]:=
\cfrac{1}{a_1+
\cfrac{1}{a_2+\dots}}.
$$
It is well known that the sequence of entries $(a_n)_{n\in\NN}$ is eventually periodic if and only if $\alpha$ is a quadratic irrational. In this case we write
$$
\alpha=a_0+[b_1,\dots,b_m,\overline{a_1,\dots,a_n}],
$$
where $b_1,\dots,b_m$ is the pre-periodic part of the sequence of entries and $a_1,\dots,a_n$ is the period of the periodic part. Consider the following tree positive values $\phi_1<\phi_2<\phi_\infty$, where
\begin{eqnarray*}
&&
\phi_1:=7+14\cdot[\overline{3,1}]=
\frac{7\sqrt{21}}{3}=
10,696277\pm10^{-6}
\\
&&
\phi_2:=
14\cdot[1,4,\overline{1,3}]=
14\cdot\frac{4\sqrt{21}+18}{5\sqrt{21}+21}=
11,582576\pm10^{-6}
\\
&&
\phi_\infty:=14\cdot[1,4,\overline{1,4,2,4}]=
14\cdot\frac{2\sqrt{210}+24}{2\sqrt{210}+35}=
11,593101\pm10^{-6}.
\end{eqnarray*}

Consider a bounded interval $G$ and denote its endpoints by $G^{(-)}:=\inf G$ and $G^{(+)}:=\sup G$. A gap in $\cL(S)$ is an open interval $G$ as above such that $G\cap\cL(S)=\emptyset$ and $G^{(-)}\in\cL(S)$ and $G^{(+)}\in\cL(S)$. In particular, if $G$ is a gap in $\cL(S)$ then its endpoints are elements of the spectrum.

\begin{theorem}
\label{TheoremMainTheorem}
Let $S$ be any origami in the orbit B7 in $\cH(2)$ and consider its Lagrange spectrum $\cL(S)$. The following holds
\begin{enumerate}
\item
We have $\min\cL(S)=\phi_1$ and the latter is an isolated point of $\cL(S)$. More precisely $G_0:=(\phi_1,\phi_2)$ is the first gap of $\cL(S)$.
\item
There exists a sequence of gaps $(G_k)_{k\geq0}$ such that
\begin{eqnarray*}
&&
\phi_1=G_{0}^{(-)}< \dots <G_{k}^{(-)}< G_{k}^{(+)}<G_{k+1}^{(-)}< G_{k+1}^{(+)}< \dots <\phi_\infty
\textrm{ for any }k\geq0
\textrm{ and }
\\
&&
G_{k}^{(+)}\to\phi_\infty
\textrm{ for }k\to\infty.
\end{eqnarray*}
\item
For any $k\geq1$ there exists a sequence of gaps $G_{k,n}$ with $n\geq1$ such that
\begin{eqnarray*}
&&
G_{k}^{(+)}<\dots < G_{k,n+1}^{(-)}< G_{k,n+1}^{(+)}<G_{k,n}^{(-)}< G_{k,n}^{(+)}< \dots <\phi_\infty
\textrm{ for any }n
\textrm{ and }
\\
&&
G_{k,n}^{(-)}\to G_{k}^{(+)}
\textrm{ for }n\to\infty.
\end{eqnarray*}
In particular $\phi_2$ is the second smallest value of $\cL(S)$ and it is not an isolated point of the spectrum.
\item
There exists a gap $G_\infty$ in $\cL(S)$ such that $G_\infty^{(-)}=\phi_\infty$. 
\end{enumerate}
\end{theorem}

Consider the quadratic irrational
$$
\eta_1=
7\cdot\frac{[1,4,2,\overline{1,5}]+5+[1,5,1,\overline{1,5}]}{4}
=
11,655309\pm10^{-6}
$$
and observe that $\eta_1>\phi_\infty$. The main intermediate result in the proof of Theorem \ref{TheoremMainTheorem} is Theorem \ref{TheoremReductionToSubshift} below.

\medskip

Consider the two finite words $a:=1,4,2,4$ and $b:=1,3$. Let $\Xi$ be the subset of $\{a,b\}^\ZZ$ of those sequences $\xi=(\xi_k)_{k\in\ZZ}$ such that for any $N$ there exists $n>N$ wit $\xi_n=a$. Let $\Xi_0\subset\Xi$ be the subset of those  $\xi\in\Xi$ satisfying the extra condition $\xi_0=a$. Let $\tilde{\sigma}:\Xi\to\Xi$ be the shift and consider its first return map $\sigma:\Xi_0\to\Xi_0$. If $\omega$ is a positive infinite word in the letters $1,2,3,4$ let $[\omega]\in(0,1)$ be the real number $\alpha$ such that the entries $a_n$ of the continued fraction $\alpha=[a_1,a_2,\dots]$ are the letters of $\omega$. Define two functions $[\cdot]_{+}:\Xi_0\to\RR$ and $[\cdot]_{-}:\Xi_0\to\RR$ by
\begin{eqnarray*}
&&
[\xi]_{+}:=[1,4,\xi_1,\xi_2,\dots]
\\
&&
[\xi]_{-}:=[1,4,\xi_{-1},\xi_{-2},\dots].
\end{eqnarray*}
Define a function $L^\sigma:\Xi_0\to\RR_+$ by
$$
L^\sigma(\xi):=
7\cdot
\big(
\limsup_{n\to+\infty}
[\sigma^n(\xi)]_{-}+[\sigma^n(\xi)]_{+}
\big).
$$

\begin{theorem}
\label{TheoremReductionToSubshift}
Consider data $(X_j,\alpha)$, where $X_j$ is an origami in the orbit B7 and $\alpha\in\RR$. If $L^\sigma(X_j,\alpha)>\phi_1$ then we have also $L(X_j,\alpha)\geq \phi_2$. Moreover, for those data $(X_j,\alpha)$ with
$
\phi_2\leq L(X_j,\alpha)<\eta_1
$
there exists $\xi\in\Xi_0$ such that
$$
L(X_j,\alpha)=L^\sigma(\xi).
$$
\end{theorem}

Thus, Theorem \ref{TheoremReductionToSubshift} shows that the study of the bottom of the Lagrange spectrum of origamis in the Orbit B7 can be reduced to the study of the values of the function $L^\sigma$ on a subshift of finite type. More precisely, let $\KK$ be the set of values of the function $L^\sigma:\Xi_0\to\RR_+$. By Theorem \ref{TheoremReductionToSubshift}, $\cL(S) \cap [ \phi_2, \eta_1) = \KK$ for any origami $S$ in the Orbit B7. In this paper we could not completely describe the set $\KK$, nevertheless we observed a certain autosimilarity in it and we leave the following as open problems.

\begin{conjecture}\label{Q1}
We conjecture that $\KK$ is a Cantor set. The families of gaps $(G_k)_{k\geq2}$
and $(G_{k,n})_{k\geq2,n\geq1}$ in Theorem \ref{TheoremMainTheorem} seem to give respectively the first and the second generation of an iterative construction.
\end{conjecture}

\begin{question}\label{Q2} 
Does it exists a constant $c>0$ such that for any real number $\phi$ with $\phi_2<\phi<\phi_\infty$ we have for the Hausdorff dimension
$$
HD\big(\KK\cap(\phi_2,\phi)\big)>c?
$$
\end{question}

\subsection*{Structure of the rest of the paper}

In \S~\ref{Sec:FormulaWithRenormalization}, the main result is Theorem \ref{thm:renormalizedformula}, which holds for any primitive square-tiled surface $S$ and gives a formula to compute $L(S,\alpha)$ in terms of the continued fraction expansion of $\alpha$. In general, our formula can be applied only for slopes $\alpha$ such that $L(\TT^2,\alpha)$ is bigger that a given value, which depends only on the orbit of $S$ under $\sltwoz$. The starting point is Lemma \ref{Lem:Lem5.10HubertMarcheseUlcigrai}, which expresses $L(S,\alpha)$ in terms of small values of the quantity
$
m^2(p/q,S)\cdot q\cdot|q\alpha-p|
$
when $p/q$ varies among rational numbers. The factor $m(p/q,S)$ is called \emph{multiplicity} of the slope $p/q$ and is defined in \S~\ref{Sec:MultiplicityRationalDirections}, where we also state and prove the covariance of the multiplicity for primitive origamis in the same orbit under $\sltwoz$. Some general facts about $\sltwoz$-orbits are recalled in \S~\ref{sec:action-sltwoz-primitive-origamis}.

In \S~\ref{sec:formula-for-orbit-B7} we state Corollary \ref{cor:renormalizedformulafororbitB7}, which says that for the origamis $S$ in the orbit B7 of $\cH(2)$ the formula of Theorem \ref{thm:renormalizedformula} gives the value $L(S,\alpha)$ \emph{for any} $\alpha$. In \S~\ref{Sec:DescriptionOfOrbitB7} the orbit B7 of $\cH(2)$ is described explicitly. In \S~\ref{Sec:PeriodicElements} we consider a quadratic irrational $\alpha$, or equivalently a closed geodesics in moduli space, and we establish a formula, namely Equation \eqref{eqFormulaLagrangeConstantQuadrIrrat}, which gives $L(S,\alpha)$ as the maximum over a finite set of values. This is very useful for a numerical study of the Lagrange spectrum, which we do not attempt in this paper.

In \S~\ref{Sec:ProofOfTheoremReductionToSubshift} we prove Theorem \ref{TheoremReductionToSubshift}. The orbit B7 has the structure of an oriented graph, whose arcs represent the action of the two generators $T$ and $R$ of $\sltwoz$. We introduce the \emph{intermediate graph} $\cI$, whose vertices are 8 special elements of the orbit B7 and whose arrows are 30 special combinations of the operations $T$ and $R$, and its subgraph $\cS$, called \emph{small graph}, which has 3 vertices and $6$ arrows. Proposition \ref{PropReductionToIntermediateGraph} says that the renormalization of data $(S,\alpha)$ with $L(S,\alpha)<\eta_3$ only contains operations appearing in the intermediate graph, where $\eta_3$ is a quadratic irrational such that $\eta_3>\eta_1>\phi_\infty$. Then Proposition \ref{propReductionToSmallgraph} reduces the possible operations to the small graph under the condition $L(S,\alpha)<\eta_1$. After such drastic reduction of the number of possible operations, in \S~\ref{Sec:EndProofTheoremREductionToSubshift} we complete the proof of Theorem \ref{TheoremReductionToSubshift}, with arguments similar in the spirit to the analysis given by Cusick and Flahive in \S~1 of \cite{cusicflahive} for the classical Lagrange spectrum.

In \S~\ref{Sec:ProofOfMainTheorem} we prove Theorem \ref{TheoremMainTheorem}. In \S~\ref{Sec:LexicographicOrder} we introduce a lexicographic order on half-infinite words in the letters $a,b$, where $a=1,4,2,4$ and $b=1,3$. Such order enables to establish an order on the values of the function $L^\sigma:\Xi_0\to\RR_+$. Since we believe that the set $\KK$ of such values is a Cantor set,  we present the proof as an iterative construction of what we believe are the first two levels of a Cantor set. The first level is given in \S~\ref{SecFirstGenerationCantor} and the second level is given in \S~\ref{SecSecondGenerationCantor}. Finally in \S~\ref{EndOfTheProofOfMainTheorem} we conclude with the proof Theorem \ref{TheoremMainTheorem}, which essentially follows simply by rephrasing the previous results.

\section{A formula in terms of the action of $\sltwoz$ }
\label{Sec:FormulaWithRenormalization}

\subsection{Continued fraction and $\sltwoz$}
\label{Sec:ContinuedFractionAndSL(2,Z)}

Consider an irrational number
$
\alpha=a_0+[a_1,a_2,\dots]
$,
where $a_0$ is the integer part of $\alpha$ and the sequence $(a_n)_{n\in\NN^\ast}$ of positive integers corresponds to the fractional part. For any $n\geq1$ we set
$$
[a_1,\dots,a_n]:=
\cfrac{1}{a_1+
\cfrac{1}{a_2+\dots+\cfrac{1}{a_n}}}.
$$
The $n$-th \emph{Gauss approximation} of $\alpha$ is
$
p_n/q_n:=a_0+[a_1,a_2,\dots,a_n]
$.
For any such $n$, the \emph{intermediate Farey approximations} are
$$
\frac{p_{n,i}}{q_{n,i}}:=
a_0+[a_1,a_2,\dots,a_{n-1},i]
\textrm{ with }
1\leq i< a_n.
$$

Consider the action of $\sltwoz$ on $\RR$ by homographies, that is
$$
\begin{pmatrix}
a & b\\
c & d
\end{pmatrix}
\alpha:=
\frac{a\alpha+b}{c\alpha+d}.
$$
Consider the elements
$$
T:=
\begin{pmatrix}
1 & 1\\
0 & 1
\end{pmatrix}
\textrm{ , }
V:=
\begin{pmatrix}
1 & 0\\
1 & 1
\end{pmatrix}
\textrm{ and }
R:=
\begin{pmatrix}
0 & -1\\
1 & 0
\end{pmatrix}
$$
and recall that $\{T,V\}$ are a set of generators, so that also $\{T,R\}$ generate, observing that
$$
V=R\circ T^{-1}\circ R^{-1}.
$$

The following Lemma holds both for Gauss and Farey approximations. For simplicity we state it just for the former.

\begin{lemma}\label{lemslopesandpathsinSL(2,Z)}
If $\alpha=a_0+[a_1,a_2,\dots]$ then the sequence of co-slopes $p_n/q_n$ of the Gauss approximations of $\alpha$ is given by
\begin{eqnarray*}
&&
p_n/q_n=
T^{a_0}\circ V^{a_1}\circ\dots\circ
V^{a_{n-1}}\circ T^{a_n}\cdot 0,
\textrm{ for even }
n
\\
&&
p_n/q_n=
T^{a_0}\circ V^{a_1}\circ\dots\circ
T^{a_{n-1}}\circ V^{a_n}\cdot \infty
\textrm{ for odd }
n.
\\
\end{eqnarray*}
\end{lemma}

\begin{proof}
Just recall that the recursive relations satisfied by the convergents show that the sequence $(p_n,q_n)$  is obtained by setting $(p_{-2},q_{-2})=(0,1)$ and $(p_{-1},q_{-1})=(1,0)$ and then applying for any $k\in\NN$ the recursive relations
$$
\begin{pmatrix} p_{2k-1} & p_{2k}\\
q_{2k-1} & q_{2k}
\end{pmatrix}=
\begin{pmatrix}
p_{2k-1} & p_{2k-2}\\
q_{2k-1} & q_{2k-2}
\end{pmatrix}
\circ T^{a_{2k}}
\textrm{ and }
\begin{pmatrix}
p_{2k+1} & p_{2k}\\
q_{2k+1} & q_{2k}
\end{pmatrix}=
\begin{pmatrix}
p_{2k-1} & p_{2k}\\
q_{2k-1} & q_{2k}
\end{pmatrix}
\circ V^{a_{2k+1}}.
$$
\end{proof}

\subsection{Action of $\sltwoz$ on primitive origamis}
\label{sec:action-sltwoz-primitive-origamis}

The \emph{Veech group} of a translation surface $S$ is the subgroup $\veech(S)$ of $\sltwor$ of those $G$ such that $G\cdot S=S$. It is easy to see that $\veech(S)$ is never co-compact and it is also well-known that the quotient $\sltwor/\veech(S)$ has finite volume if and only if the orbit $\sltwor\cdot S$ is closed in its stratum. This is always true for origamis. We say that a square-tiled surface $S$ is \emph{primitive} if $\langle\hol(S)\rangle=\ZZ^2$, that is the set of periods of $S$ generates $\ZZ^2$ as subgroup of $\RR^2$. This implies that the Veech group $\veech(S)$ of $S$ is a finite-index subgroup of $\sltwoz$. For any origami $S$ the veech group $\veech(S)$ and $\ZZ^2$ share a common subgroup of finite index (see \cite{hubertlelievre}). The action of $\sltwor$ on translation surfaces induces an action of $\sltwoz$ on square tiled surfaces. If $S$ is any square tiled surface, denote $\cO(S)$ its orbit under $\sltwoz$, that is
$$
\cO(S):=\{Y=A\cdot S\textrm{ ; }A\in\sltwoz\}
$$
The number of squares $N$ of an origami $S$ is obviously preserved under the action of $\sltwoz$.

\begin{lemma}[Hubert-Leli\`evre, \cite{hubertlelievre}]\label{lemhubertlelievre}
The $\sltwoz$-orbit $\cO(S)$ of a primitive square-tiled surface $S$ with $N$ squares is the set of primitive square-tiled surfaces with $N$ squares in its $\sltwor$-orbit.
\end{lemma}

According to Lemma \ref{lemhubertlelievre} above, the set of primitive origamis is preserved under the action of $\sltwoz$, moreover each orbit $\cO(S)$ is finite and we have the identification
$$
\cO(S)=\sltwoz/\veech(S).
$$
The action of $\sltwoz$ passes to the quotient $\cO(S)$. We chose the generators $\{T,R\}$ for $\sltwoz$. In terms of $T$ and $R$, this action is represented by a oriented graph $\cG(S)$ whose vertices are the elements of $\cO(S)$ and whose oriented edges correspond to the operations $Y\mapsto T\cdot Y$ and $Y\mapsto R\cdot Y$ for $Y\in \cO(S)$. Orbits of primitive origami in $\cH(2)$ with prime number of squares $N$ have been classified in \cite{hubertlelievre}, according to the number of integer \emph{Weierstrass points}, then the classification was extended to the case of non prime $N$ in \cite{mcmullen}. For any $N=3$ and any even $N$ there exists only one orbit. For any odd $N\geq5$ there are two orbits. The first, called AN, contains primitive origamis with only one integer Weierstrass point. The second, called BN, contains primitive origamis with exactly three integer Weierstrass points. We refer to \cite{hubertlelievre} and \cite{mcmullen} for more details. The complete description of the orbit B7 is given in \S~\ref{Sec:DescriptionOfOrbitB7}, which is all we need in this paper.

\subsubsection{Cusps}

Let $S$ be a primitive origami. The cusps of $\HH/\veech(S)$ correspond to conjugacy classes under $\veech(S)$ of its \emph{primitive parabolic elements}, that is the elements in $\veech(S)$ with trace equal to $\pm2$, where primitive means not powers of other parabolic elements of $\veech(S)$. If $S$ is a primitive origami the eigendirections of parabolic elements of $\veech(S)$ are exactly the elements of $\QQ$. Therefore the cusps of $\HH/\veech(S)$ correspond to equivalence  classes for the homographic action $p/q\mapsto A\cdot p/q$ of $\veech(S)$ on $\QQ$. Lemma \ref{lemzorich} below gives a representation of cusps in terms of the action of $\sltwoz$ (a proof can be found in \cite{hubertlelievre}).

\begin{lemma}[Zorich]\label{lemzorich}
Let $S$ be a reduced origami. Then the cusps of $\HH/\veech(S)$ are in bijection with the $T$-orbits in $\cO(S)$.
\end{lemma}

\subsection{Multiplicity of a rational direction}
\label{Sec:MultiplicityRationalDirections}

Fix a primitive square-tiled surface $S$ and denote $\rho:S\to\TT^2$ the ramified covering onto the standard torus. If $\gamma:[0,1]\to S$ is a saddle
connection for $S$, we define its \emph{multiplicity} $m(\gamma)$ as the degree of
the map $t\mapsto \rho\circ\gamma(t)$. The \emph{multiplicity of a rational direction with co-slope} $p/q$ \emph{over the surface} $S$ is the minimal
multiplicity among all saddle connections on $S$ with the same co-slope $p/q$, that is the number $m(p/q;S)$ defined by
$$
m(p/q;S):=
\min
\{m(\gamma); \quad \gamma \  \textrm{saddle connection with }\ \hol(\gamma)\wedge(p+iq )=0\}.
$$

The multiplicity is \emph{covariant} under the left action of $\sltwoz$, that is
\begin{equation}\label{eqcovariancemultiplicity}
m(p/q;S)=m(A\cdot p/q;A\cdot S).
\end{equation}

\begin{remark}
If the co-slopes $p/q$ and $p'/q'$ are in the same cusp then there exists some $A\in\veech(S)$ such that
$
p'/q'=A\cdot(p/q)
$.
Thus Equation \eqref{eqcovariancemultiplicity} implies
$$
m(p/q;S)=m(p'/q';S).
$$
\end{remark}

Fix $n\in\NN$ and consider positive integers $a_1,\dots,a_n$. Define the element $g(a_1,\dots,a_n)$ of $\sltwoz$ by
\begin{eqnarray*}
&&
g(a_1,\dots,a_n):=
(T^{-a_n}R)\dots(T^{-a_2}R)(T^{a_{1}}R)
\textrm{ if }
n
\textrm{ is even }\\
&&
g(a_1,\dots,a_n):=
(T^{a_n}R)\dots (T^{-a_2}R)(T^{a_{1}}R)
\textrm{ if }
n
\textrm{ is odd }.
\\
\end{eqnarray*}

\begin{lemma}
\label{lemMultiplicityAlongPaths}
For any finite sequence $a_1,\dots,a_n$ we have
$$
m\big([a_1,\dots,a_n];S\big)=
m\big(\infty;
R\cdot g(a_1,\dots,a_n)\cdot S
\big).
$$
\end{lemma}

\begin{proof}
Recall that projectively we have $R^2=\textrm{Id}$, that is $R=R^{-1}$. Observe that $\infty=R\cdot 0$. According to Lemma \ref{lemslopesandpathsinSL(2,Z)} any rational number $[a_1,\dots,a_n]$ in $(0,1)$ can be written as
$$
[a_1,\dots,a_n]=
g(a_1,\dots,a_n)^{-1}\cdot R\cdot\infty.
$$
Write for simplicity $g:=g(a_1,\dots,a_n)$. The Lemma follows from the covariance of multiplicity stated by Equation \eqref{eqcovariancemultiplicity}, indeed we have
$$
m([a_1,\dots,a_n];S)=
m(g^{-1}\cdot R\cdot\infty;S)=
m(g^{-1}\cdot R\cdot\infty;g^{-1}\cdot R^2\cdot g \cdot S)=
m(\infty;R\cdot g\cdot S).
$$
\end{proof}

\subsection{Selection of relevant rational approximations}

Denote $\alpha=[a_1,a_2,\dots]$ the continued fraction expansion of $\alpha\in(0,1)$ and recall that for any $n$ and any $i$ with $1\leq i<a_n$ we set
$$
p_n/q_n:=[a_1,a_2,\dots,a_n]
\textrm{ and }
\frac{p_{n,i}}{q_{n,i}}:=
[a_1,a_2,\dots,a_{n-1},i]
$$

The following is a very classical result, nevertheless we provide a proof for completeness.

\begin{lemma}\label{lem1:renormalizedformula}
Fix $\alpha=[a_1,a_2,\dots]$. For any $n$ we have
$$
\frac{1}{q_n\cdot|q_n\alpha-p_n|}=
[a_{n},\dots,a_1]+a_{n+1}+[a_{n+2},a_{n+3},\dots]
$$
For any $n$ and any $i$ with $1\leq i< a_n$ we have
$$
\frac{1}{q_{n,i}\cdot|q_{n,i}\alpha-p_{n,i}|}=
[i,\dots,a_1]+[a_{n}-i,a_{n+1},\dots].
$$
\end{lemma}

\emph{Note:} For any $n$ and any $i$ with $1\leq i< a_n$ we have
$
[i,\dots,a_1]+[a_{n}-i,a_{n+1},\dots]<2
$,
which corresponds to the well-known fact that $q|q\alpha-p|<1/2$ just for $p/q=p_n/q_n$.

\medskip

\begin{proof}
Suppose that $(p,q)$ and $(p',q')$ form a basis of $\ZZ^2$, so that $qp'-pq'=\pm1$, and assume also that $(\alpha,1)$ belongs to the convex cone spanned by these two vectors, that is 
$
(q\alpha-p)(q'\alpha-p')<0
$. 
Then we have
$$
\frac{1}{q|q\alpha-p|}=
\left|\frac{qp'-pq'+qq'\alpha-qq'\alpha}{q(q\alpha-p)}\right|=
\frac{q'}{q}+\frac{p'-q'\alpha}{q\alpha-p}.
$$
In order to prove both the two parts of the statement we set $q':=q_{n-1}$ and $p':=p_{n-1}$. Observe that for any $n$ and any $i$ with $1\leq i\leq a_n$, thus both for Farey and Gauss approximations, we have
$$
\frac{q_{n-1}}{q_{n,i}}=[i,a_{n-1},\dots,a_1].
$$
To simplify the notation, for any $n$ and any $i$ set
$
l_{n,i}:=|q_{n,i}\alpha-p_{n,i}|
$.
Set also $l_n:=|q_{n}\alpha-p_{n}|$. The first part of the statement follows observing that
$
l_{n-1}=a_{n+1}\cdot l_n+l_{n+1}
$,
so that
$$
\frac{l_{n-1}}{l_n}=a_{n+1}+[a_{n+2},a_{n+2},\dots].
$$
The second part of the statement follows observing that
$
l_{n,i}=l_n+(a_n-i)\cdot l_{n-1}
$,
so that
$$
\frac{l_{n-1}}{l_{n,i}}=
\cfrac{1}{a_n-i+\cfrac{l_n}{l_{n-1}}}=
[a_{n}-i,a_{n+1},\dots].
$$
\end{proof}

\begin{lemma}\label{lem2:renormalizedformula}
Let $\alpha$ be an irrational slope. For any $\epsilon>0$ there exists $Q>0$ such that for any rational $p/q$ with $q>Q$ and which is neither a Gauss approximation of $\alpha$ nor a Farey approximation we have
$$
q\cdot|q\alpha-p|\geq
1+2\bigg(\frac{1}{L(\TT^2,\alpha)}-\epsilon\bigg).
$$
\end{lemma}

\begin{proof}
We assume $p/q<\alpha$, the other case being the same, then consider $n$ corresponding to those Gauss approximations such that 
$
p_n/q_n<\alpha<p_{n-1}/q_{n-1}
$. 
Assume also $a_n\geq2$, the lattice argument for the case $a_n=1$ being the same. According to the assumption in the statement, let $p_{n,i}/q_{n,i}$ and $p_{n,i+1}/q_{n,i+1}$ be two consecutive Farey approximations of $\alpha$ such that we have the strict inequality
$$
p_{n,i}/q_{n,i}<p/q<p_{n,i+1}/q_{n,i+1}<\alpha.
$$

\begin{figure}[h]
\centering
\includegraphics[width=0.6\textwidth]{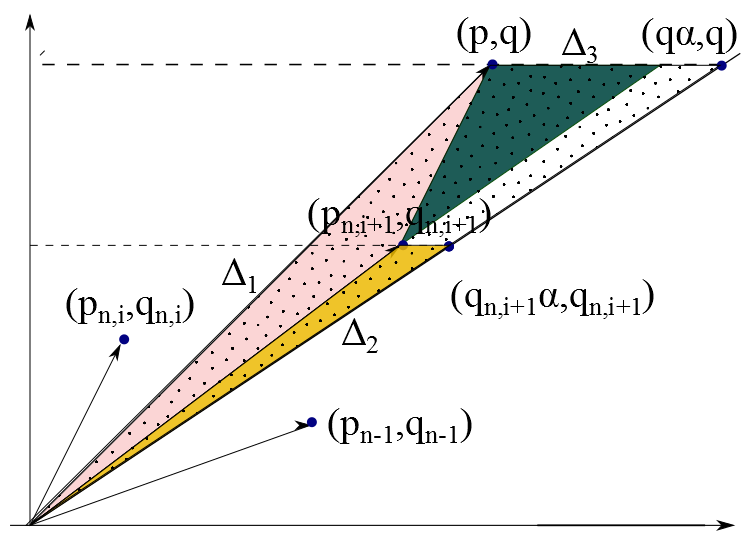}
\caption{Illustration for the proof of Lemma \ref{lem2:renormalizedformula}.}
\label{triangles}
\end{figure}

We have
$
(p_{n,i+1},q_{n,i+1})=
(p_{n,i},q_{n,i})+(p_{n-1},q_{n-1})
$.
Moreover the integer vectors $(p_{n-1},q_{n-1})$ and $(p_{n,i},q_{n,i})$ form a basis of $\ZZ^2$, therefore we have a decomposition
$$
(p,q)=
x(p_{n-1},q_{n-1})+
y(p_{n,i},q_{n,i})
\textrm{ with }
x\geq1
\textrm{ and }
y\geq2.
$$
Condition $y\geq2$ holds because the linear combinations with $y=1$ correspond to Farey's approximations. One can see that the triangle with vertices 
$(p,q)$, $(q\alpha,q)$ and $(0,0)$ (which is dotted in Figure \ref{triangles}) contains the union of the triangles $\Delta_1$, $\Delta_2$ and $\Delta_3$ (see Figure \ref{triangles}), which are disjoint in the interior, where $\Delta_1$ is the triangle with vertices $(p,q)$, $(p_{n,i+1},q_{n,i+1})$ and $(0,0)$, $\Delta_2$ is the triangle with vertices $(p_{n,i+1},q_{n,i+1})$, $(q_{n,i+1}\alpha,q_{n,i+1})$ and $(0,0)$ and $\Delta_3$ is the triangle obtained translating by the vector $(p_{n,i+1},q_{n,i+1})$ the triangle with vertices $(p_{n,i},q_{n,i})$, $(q_{n,i}\alpha,q_{n,i})$ and $(0,0)$. One can see that the area of the big triangle (dotted in Figure) is $q|q\alpha-p|$  and by  computing the areas $A(\Delta_i)$ of the triangles one has that $A(\Delta_1)=|qp_{n,i+1}-q_{n,i+1}p|$, $A(\Delta_2)=|
q_{n,i+1}|q_{n,i+1}\alpha-p_{n,i+1}|$ and $A(\Delta_3)=|
q_{n,i}|q_{n,i}\alpha-p_{n,i}|$. 
Therefore we have
$$
q|q\alpha-p|\geq
|qp_{n,i+1}-q_{n,i+1}p|+
q_{n,i}|q_{n,i}\alpha-p_{n,i}|+
q_{n,i+1}|q_{n,i+1}\alpha-p_{n,i+1}|.
$$
Fix $\epsilon>0$ and set $a:=(L(\TT^2,\alpha))^{-1}$. If $n$ is big enough then we have both
$
q_{n,i}|q_{n,i}\alpha-p_{n,i}|>a-\epsilon
$
and
$
q_{n,i+1}|q_{i+1}\alpha-p_{i+1}|>a-\epsilon
$, 
hence 
$
q|q\alpha-p|>1+2(a-\epsilon)
$.
\end{proof}

\subsection{The formula with continued fraction}

Fix a primitive origami $S$ and let $N_S$ be the number of squares of $S$ and
$$
M_S:=\max_{p/q\in\QQ}
m(p/q;S).
$$
Consider irrational slopes $\alpha=[a_1,a_2,\dots]$ in $(0,1)$. Recall that the function $L(\TT^2,\alpha)$ is invariant under the homographic action of $\sltwoz$. More generally, the function $\alpha\mapsto L(S,\alpha)$ is invariant under the Veech group $\veech(S)$ of $S$. Since the latter acts expansively and transitively on $\RR\cup\{\infty\}$, then in order to compute $\cL(S)$ it is enough to consider $\alpha\in(0,1)$, that is
$$
\cL(S)=\{L(S,\alpha);0<\alpha<1\}.
$$
Recall Lemma 5.10 in \cite{hubertmarcheseulcigrai}.

\begin{lemma}
\label{Lem:Lem5.10HubertMarcheseUlcigrai}
Consider a primitive origami $S$. We have
$$
L(S,\alpha):=
N_S\cdot
\limsup_{q,p\to\infty}
\frac{1}{m^2(p/q,S)\cdot q\cdot|q\alpha-p|}.
$$
\end{lemma}

\medskip

For any positive integer $n$ and any integer $i$ with $1\leq i\leq a_n$ set
\begin{eqnarray*}
&&
D(n,i,\alpha):=
[a_{n},\dots,a_1]+a_{n+1}+[a_{n+2},a_{n+3},\dots]
\textrm{ if }
i=a_n\\
&&
D(n,i,\alpha):=
[i,\dots,a_1]+[a_{n}-i,a_{n+1},\dots]
\textrm{ if }
1\leq i<a_n.\\
\end{eqnarray*}

\begin{theorem}\label{thm:renormalizedformula}
Let $S$ be a reduced origami and let
$
\alpha=[a_1,a_2,\dots]\in(0,1)
$
be an irrational slope such that
$$
L(\TT^2,\alpha)>M_S^2-2.
$$
Then we have
$$
L(S,\alpha)=
N_S\cdot
\limsup_{n\to\infty}
\max_{1\leq i\leq a_n}
\frac{D(n,i,\alpha)}
{m^2\big(\infty;R\cdot g(a_1,\dots,a_{n-1},i)\cdot S\big)}.
$$
\end{theorem}

\begin{proof}
Set
$
a:=L(\TT^2,\alpha)^{-1}
$
and observe that the assumption in the statement is equivalent to
$
(M_S^2-2)\cdot a<1
$.
Therefore consider $\epsilon>0$ such that
$$
1+2(a-\epsilon)>
M_S^2\cdot(a+\epsilon).
$$
For any rational $p/q$ with $q$ big enough and which is neither a Farey nor a Gauss approximation of $\alpha$ then Lemma \ref{lem2:renormalizedformula} implies
$$
q\cdot|q\alpha-p|>1+2(a-\epsilon).
$$
On the other hand there exists infinitely many $n$ such that
$
q_n\cdot|q_n\alpha-p_n|<a+\epsilon
$,
therefore we have
$$
\liminf_{n\to\infty}
m^2(p_n/q_n;S)\cdot
q_n\cdot|q_n\alpha-p_n|<
M_S^2(a+\epsilon).
$$
It follows that
$
\liminf_{q,p\to\infty}
m^2(p/q;S)\cdot q\cdot|q\alpha-p|
$
is taken either along the sequence of Farey approximations, or along the subsequence of Gauss approximation. The Theorem follows from the formulae in Lemma \ref{lem1:renormalizedformula}.
\end{proof}

\begin{remark}
Theorem \ref{thm:renormalizedformula} generalizes the nice classical formula
$$
L(\TT^2,\alpha):=
\limsup_{n\to\infty}
[a_{n},\dots,a_1]+a_{n+1}+[a_{n+2},a_{n+3},\dots].
$$
\end{remark}

\section{The formula for the orbit B7}
\label{sec:formula-for-orbit-B7}

\subsection{Description of the orbit B7 in $\cH(2)$}
\label{Sec:DescriptionOfOrbitB7}

Let $\cO$ be the B-orbit of primitive square tiled surfaces with $7$ squares. The orbit contains $36$ elements, partitioned into $8$ cusps. Denote the cusps by the letters
$$
A,B,C,D,E,F,G,H.
$$
For each cusp $X$ denote $w=w(X)$ its width, so that
$$
w(A)=w(B)=w(C)=7
\textrm{ , }
w(D)=w(G)=w(H)=3
\textrm{ , }
w(E)=5
\textrm{ and }
w(F)=1.
$$
Denote by $X_j$ the elements in the cusp $X$, where the index $j$
is an integer with $0\leq j\leq w(X)-1$, so that for example
the elements of cusp $C$ are
$C_0$, $C_1$, $C_2$, $C_3$, $C_4$, $C_5$, $C_6$. 

\begin{remark}
\label{RemWhyWeChoseOrbitB}
According to Lemma \ref{lemMultiplicityAlongPaths}, all the possible values of the function 
$$
m:\QQ\times\cO\to\NN^\ast
\textrm{ , }
(p/q,X_j)\mapsto m(p/q,X_j)
$$ 
are obtained by its restriction to $\{\infty\}\times\cO$. The co-slope $p/q=\infty$ corresponds to the horizontal direction, and it is evident from Figure \ref{fig:surfaces} that $m(\infty,X_j)\in\{1,2\}$ for any $X_j\in\cO$. It follows that for any $X_j\in\cO$ we have
$$
M_S=\max\{m(p/q,X_j),p/q\in\QQ\}=2.
$$
On the other hand for any $\alpha$ irrational we have 
$$
L(\TT^2,\alpha)\geq\sqrt{5}>2=M_S^2-2,
$$ 
therefore the formula for $L(X_i,\alpha)$ in Theorem \ref{thm:renormalizedformula} holds for any $\alpha$ irrational. This is the reason for choosing the orbit B.
\end{remark}

Denote by $\cG$ the directed graph whose vertices are the elements of $\cO$
and with labeled arrows for the action of $T$ and $R$,
as in section~\ref{sec:action-sltwoz-primitive-origamis}. The surfaces in $\cO$ are represented in figure~\ref{fig:surfaces}. The graph $\cG$ is represented in figure~\ref{fig:orbit-graph}, where the oriented arrows outside the circle represent the action of $T$, while the arrows inside the circle represent the action of $R$ and are unoriented since $R=R^{-1}$.

\begin{figure}[!ht]
\includegraphics[width=0.8\textwidth]{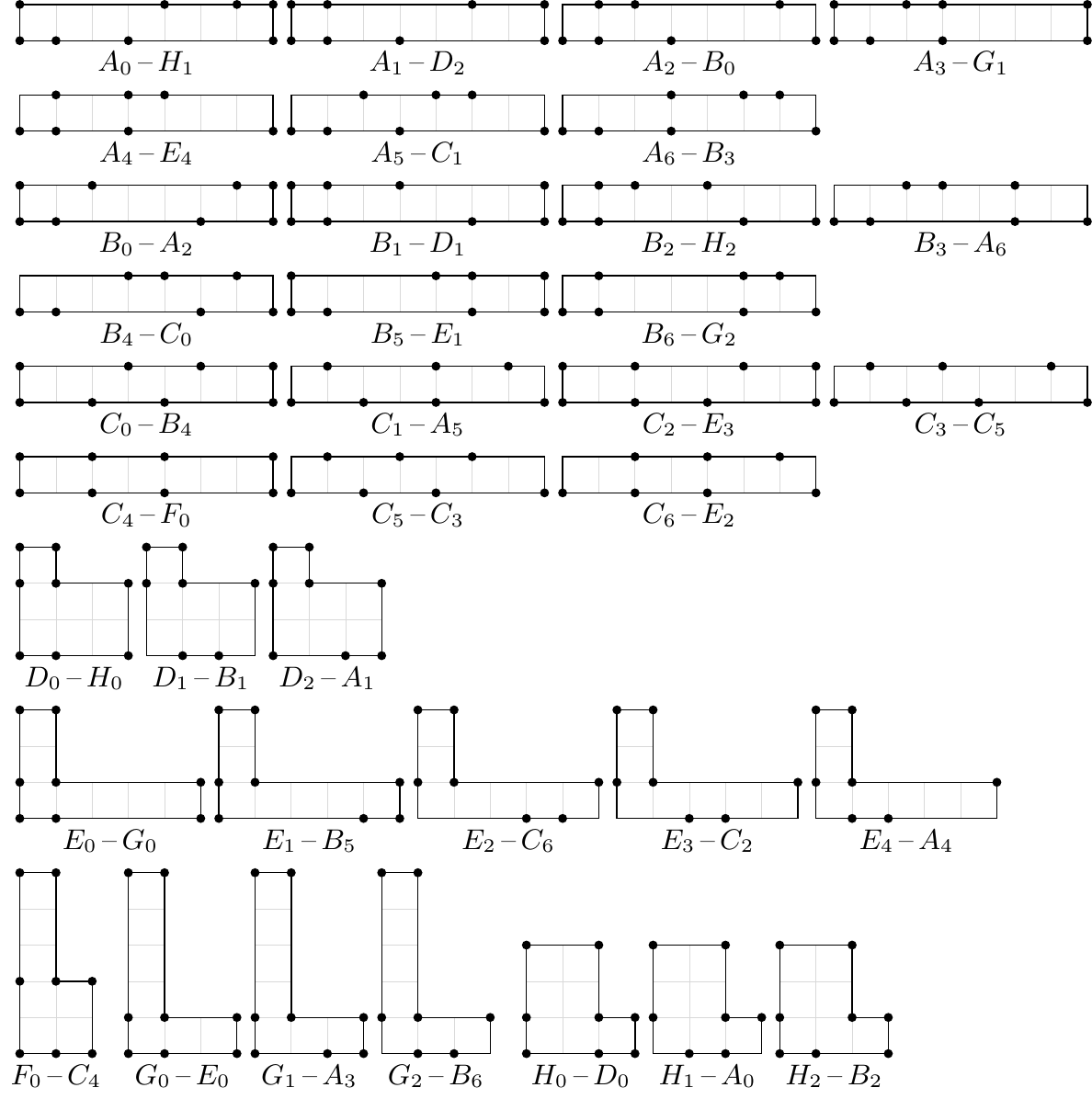}
\caption{Surfaces in orbit $\cO$, by cusp.
Under each surface, its name and the name of its image under $R$.}
\label{fig:surfaces}
\end{figure}

\begin{figure}[!ht]
\includegraphics[width=0.75\textwidth]{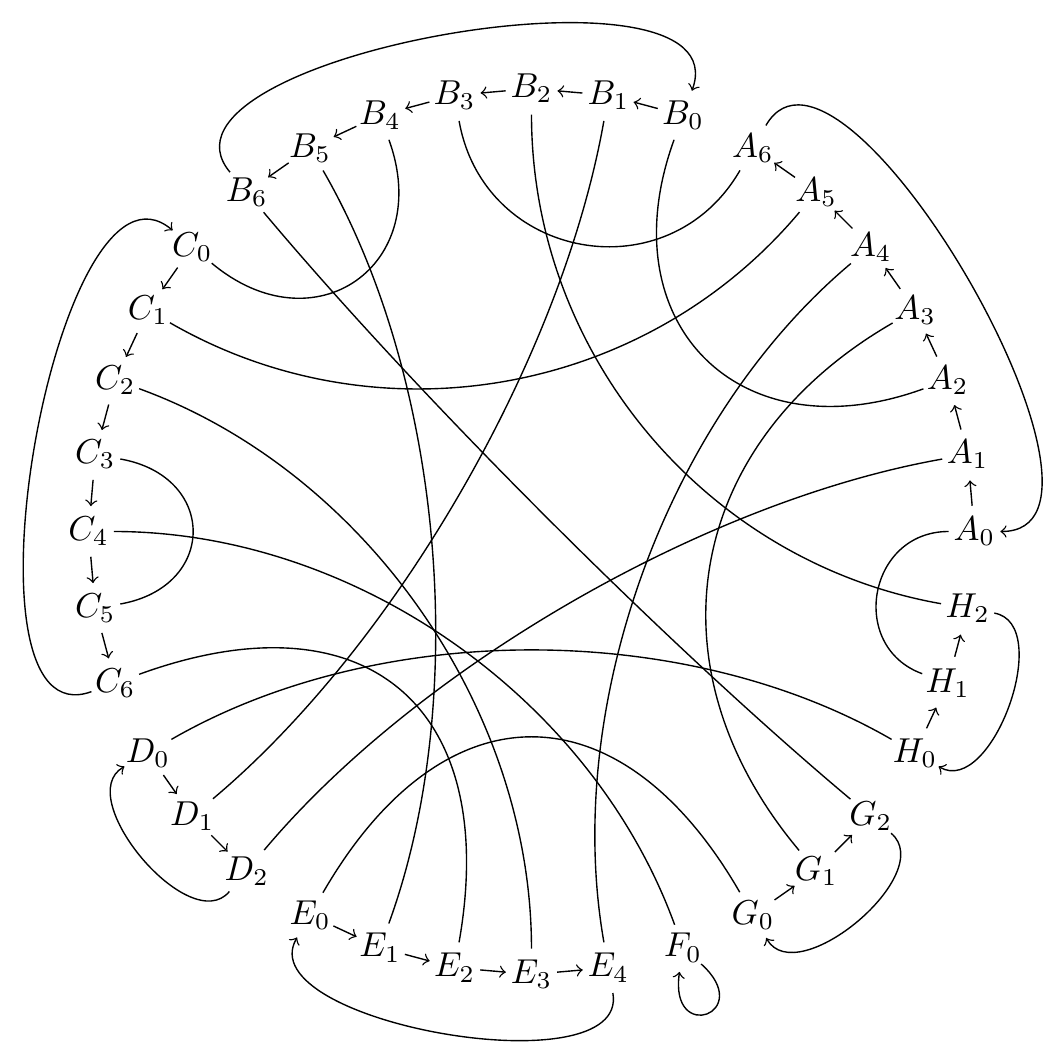}
\caption{The graph $\cG$ for orbit $\cO$, with arrows
along and outside the circle for the action of $T$,
and unoriented edges inside the circle for the (involutive) action of $R$.}
\label{fig:orbit-graph}
\end{figure}

Define a function $m:\cO\to\{1,2\}$ on the elements of $\cO$ as follows

\begin{eqnarray*}
&&
m(C_j)=2
\textrm{ for }
j=0,1,\dots,6.
\\
&&
m(X_j)=1
\textrm{ if }
X_j\not\in\{C_i;0\leq i\leq 6\}.
\\
\end{eqnarray*}

Fix $X_j$ and irrational number $\alpha=[a_0,a_1,\dots]$. The continued fraction expansion of $\alpha$ gives elements $g(a_1,\dots,a_{n-1},i)$ in $\sltwoz$, which are defined in \S~\ref{Sec:MultiplicityRationalDirections}, where $n$ is a positive integer and the integer $i$ satisfies $1\leq i\leq a_n$. They produce a path in $\cO$, denoted $(X_j,\alpha)$, whose vertices are the elements
$$
g(a_1,\dots,a_{n-1},i)\cdot X_j,
$$
spanned in order as $n$ grows and, for each fixed $n$, the index $i$ grows in the range $1\leq i\leq a_n$. More precisely, if $\alpha=[a_1,a_2,\dots]$ is the continued fraction expansion of $\alpha$, then the first $a_1$ vertices of the path  $(X_j,\alpha)$ in $\cO$ are 
$$
TR\cdot X_j,T^2R\cdot X_j,\dots,T^{a_1}R\cdot X_j,
$$
then the second $a_2$ vertices are 
$$
T^{-1}R(T^{a_1}R\cdot X_j),T^{-2}R(T^{a_1}R\cdot X_j),\dots,T^{-a_2}R(T^{a_1}R\cdot X_j)
$$ 
and so on. The vertices reached just before an arrow of type $R$, i.e. the ones of the form  $g(a_1,\dots,a_{n})\cdot X_j$, correspond to \emph{Gauss approximations}. The other vertices are of the form $g(a_1,\dots,a_{n-1},i)\cdot X_j$ for some intermediate $1\leq i <a_n$ and correspond to intermediate \emph{Farey approximations}. 

Define the support $\support(X_j,\alpha)$ as the set of those $Y_k\in\cO$ such that there exists infinitely many $n$ with
$
Y_k=g(a_1,\dots,a_n)\cdot X_j
$. 
Observe that the definition uses the Gauss approximations $[a_1,\dots,a_n]$ of $\alpha$, non just Faray approximations, thus elements of the form
$
Y_k=g(a_1,\dots,a_{n-1},i)\cdot X_j
$
for infinitely many $n$ and infinitely many $i$ with $1\leq j< a_n$ may \emph{not} belong to $\support(X_j,\alpha)$.

\subsection{The formula with continued fraction for the orbit B7}

Recall that any primitive origami $X_j$ in the orbit B7 has $N=7$ squares. As we explained in Remark \ref{RemWhyWeChoseOrbitB}, the formula for $L(X_i,\alpha)$ in Theorem \ref{thm:renormalizedformula} holds for any $\alpha=[a_0,a_1,\dots]$ irrational.

\begin{corollary}
\label{cor:renormalizedformulafororbitB7}
Let $X_j$ be any element in the orbit B7 in $\cH(2)$. Then for any irrational number $\alpha$ we have
\begin{equation}
\label{eqFormulaLagrangeConstant}
L(X_j,\alpha)=
7\cdot
\limsup_{n\to\infty}
\max_{1\leq i\leq a_n}
\frac{D(n,i,\alpha)}
{m^2\big(R\cdot g(a_1,\dots,a_{n-1},i)\cdot X_j\big)}.
\end{equation}
\end{corollary}

\subsection{Periodic elements}
\label{Sec:PeriodicElements}

An \emph{even loop} on $\cO$ is the datum $(X_j,\alpha)$, where $X_j\in\cO$ and $\alpha$ is a quadratic irrational
$
\alpha=[\overline{a_1,\dots,a_{2N}}]
$
whose period $a_1,\dots,a_{2N}$ has $2N$ entries, such that
$$
g(a_1,\dots,a_{2N})\cdot X_j=X_j.
$$

Let $\alpha=[\overline{a_1,\dots,a_{2N}}]$ be a quadratic irrational. The following is easy to prove.
\begin{enumerate}
\item
For any $n$ with $1\leq n\leq 2N$ we have
$$
D(n,a_n,\alpha)=
[\overline{a_n,\dots,a_1,a_{2N},\dots,a_{n+1}}]+
a_{n+1}+
[\overline{a_{n+2},\dots,a_{2N},a_1,\dots,a_{n+1}}]
$$
where the cyclic the order of the $2N$ entries $a_1,\dots,a_{2N}$ is inverted in the first summand and it is preserved in the third.
\item
For any $n$ with $1\leq n\leq 2N$ such that $a_n\geq 2$ and any $i$ with $1\leq i\leq a_n-1$ we have
$$
D(n,i,\alpha)=
[i,\overline{a_{n-1},\dots,a_1,a_{2N},\dots,a_{n}}]+
[a_n-i,\overline{a_{n+1},\dots,a_{2N},a_1,\dots,a_{n}}]
$$
where the cyclic order of the entries $a_1,\dots,a_{2N}$ is inverted in the period of the first summand the and it is preserved in the period of the second.
\end{enumerate}

For example, consider a quadratic irrational of the form
$
\alpha=[\overline{a_1,a_2,a_3,a_4,a_5,a_6}]
$,
where $2N=6$. For $n=3$ and $n=5$ we have respectively
\begin{eqnarray*}
&&
D(3,a_3,\alpha):=
[\overline{a_3,a_2,a_1,a_6,a_5,a_4}]+
a_4+
[\overline{a_5,a_6,a_1,a_2,a_3,a_4}]
\\
&&
D(5,a_5,\alpha):=
[\overline{a_5,a_4,a_3,a_2,a_1,a_6}]+
a_6+
[\overline{a_1,a_2,a_3,a_4,a_5,a_6}].
\\
\end{eqnarray*}
Moreover, for $n=3$, assuming that $a_3\geq 2$, for any $i$ with $1\leq i\leq a_3-1$ we have
$$
D(3,i,\alpha):=
[i,\overline{a_2,a_1,a_6,a_5,a_4,a_3}]+
[a_3-i,\overline{a_4,a_5,a_6,a_1,a_2,a_3}].
$$

For an even loop $(X_j,\alpha)$ where
$
\alpha=[\overline{a_1,\dots,a_{2N}}]
$
is a quadratic irrational whose period $a_1,\dots,a_{2N}$ has $2N$ entries Equation \eqref{eqFormulaLagrangeConstant} becomes

\begin{equation}
\label{eqFormulaLagrangeConstantQuadrIrrat}
L(X_i,\alpha)=
7\cdot
\max_{1\leq n\leq2N}
\max_{1\leq i\leq a_n}
\frac{D(n,i,\alpha)}
{m^2\big(R\cdot g(a_1,\dots,a_{n-1},i)\cdot X_j\big)}.
\end{equation}

\begin{lemma}
\label{lemMinimumAndOtherExample}
The data $(C_6,[\overline{1,3}])$ and $(C_3,[\overline{5,2}])$ define even loops in $\cO$ and we have
\begin{eqnarray*}
&&
L(C_6,[\overline{1,3}])=\phi_1=7+14\cdot[\overline{3,1}].
\\
&&
L(C_3,[\overline{5,2}])=
14\cdot[1,\overline{5,2}]=
11,832159\pm10^{-6}.
\end{eqnarray*}
\end{lemma}

\begin{proof}
One can  verify that both $(C_6,[\overline{1,3}])$ and $(C_3,[\overline{5,2}])$ are even loops with $2N=2$, which before closing up go respectively through the vertices $E_3,C_1,C_0$ and through the vertices $C_6,C_0,C_1,C_2,C_3,C_4$ (recall that the vertices of the graph after an $R$ arrow are \emph{not} recorded in the loop). Moreover, to simplify the notation, we divide everything by a factor $7$. To compute $L(C_6,[\overline{1,3}])$ we apply Equation \eqref{eqFormulaLagrangeConstantQuadrIrrat} and take the maximum among the values below
\begin{eqnarray*}
&&
\frac{D(1,1,[\overline{1,3}])}
{m^2\big(R\cdot g(1)\cdot C_6\big)}=
\frac{[\overline{1,3}]+3+[\overline{1,3}]}
{m^2\big(C_2\big)}=
\frac{[\overline{1,3}]+3+[\overline{1,3}]}{4}
=1,1456435\pm10^{-6}
\\
&&
\frac{D(2,1,[\overline{1,3}])}
{m^2\big(R\cdot g(1,1)\cdot C_6\big)}=
\frac{[1,\overline{1,3}]+[2,\overline{1,3}]}
{m^2\big(A_5\big)}=
\frac{[1,\overline{1,3}]+[2,\overline{1,3}]}{1}
=0,916515\pm10^{-6}
\\
&&
\frac{D(2,2,[\overline{1,3}])}
{m^2\big(R\cdot g(1,2)\cdot C_6\big)}=
\frac{[2,\overline{1,3}]+[1,\overline{1,3}]}
{m^2\big(B_4\big)}=
\frac{[2,\overline{1,3}]+[1,\overline{1,3}]}{1}
=0,916515\pm10^{-6}
\\
&&
\frac{D(2,3,[\overline{1,3}])}
{m^2\big(R\cdot g(1,3)\cdot C_6\big)}=
\frac{[\overline{1,3}]+1+[\overline{1,3}]}
{m^2\big(E_2\big)}=
\frac{[\overline{3,1}]+1+[\overline{3,1}]}{1}
=1,527524\pm10^{-6}.
\end{eqnarray*}
The first part is proved recalling that $\phi_1=7(1+[\overline{1,3}])$. To compute $L(C_3,[\overline{5,2}])$ we apply Equation \eqref{eqFormulaLagrangeConstantQuadrIrrat} and take the maximum among the values below
\begin{eqnarray*}
&&
\frac{D(1,1,[\overline{5,2}],C_3)}
{m^2\big(R\cdot g(1)\cdot C_3\big)}=
\frac
{[1,\overline{2,5}]+[4,\overline{2,5}]}
{m^2\big(E_2\big)}=
\frac
{[1,\overline{2,5}]+[4,\overline{2,5}]}
{1}
=0,910165\pm10^{-6}
\\
&&
\frac{D(1,2,[\overline{5,2}],C_3)}
{m^2\big(R\cdot g(2)\cdot C_3\big)}=
\frac
{[2,\overline{2,5}]+[3,\overline{2,5}]}
{m^2\big(B_4\big)}=
\frac
{[2,\overline{2,5}]+[3,\overline{2,5}]}
{1}
=0,696009\pm10^{-6}
\\
&&
\frac{D(1,3,[\overline{5,2}],C_3)}
{m^2\big(R\cdot g(3)\cdot C_3\big)}=
\frac
{[3,\overline{2,5}]+[2,\overline{2,5}]}
{m^2\big(A_5\big)}=
\frac
{[2,\overline{2,5}]+[3,\overline{2,5}]}
{1}
=0,696009\pm10^{-6}
\\
&&
\frac{D(1,4,[\overline{5,2}],C_3)}
{m^2\big(R\cdot g(4)\cdot C_3\big)}=
\frac
{[4,\overline{2,5}]+[1,\overline{2,5}]}
{m^2\big(E_3\big)}=
\frac
{[4,\overline{2,5}]+[1,\overline{2,5}]}
{1}
=0,910165\pm10^{-6}
\\
&&
\frac{D(1,5,[\overline{5,2}],C_3)}
{m^2\big(R\cdot g(5)\cdot C_3\big)}=
\frac
{[\overline{5,2}]+2+[\overline{5,2}]}
{m^2\big(C_5\big)}=
\frac
{2+2[\overline{5,2}]}
{4}
=0,591607\pm10^{-6}
\\
&&
\frac{D(2,1,[\overline{5,2}],C_3)}
{m^2\big(R\cdot g(5,1)\cdot C_3\big)}=
\frac
{[1,\overline{5,2}]+[1,\overline{5,2}]}
{m^2\big(F_0\big)}=
\frac
{2[1,\overline{5,2}]}
{1}
=1,690309\pm10^{-6}
\\
&&
\frac{D(2,2,[\overline{5,2}],C_3)}
{m^2\big(R\cdot g(5,2)\cdot C_3\big)}=
\frac
{[\overline{2,5}]+5+[\overline{2,5}]}
{m^2\big(C_5\big)}=
\frac
{5+2[\overline{2,5}]}
{4}
=1,479019\pm10^{-6}.
\end{eqnarray*}
The second is proved observing that
$
14\cdot[1,\overline{5,2}]
=
7\cdot1,690309\pm10^{-6}
$.
\end{proof}

\section{Proof of Theorem \ref{TheoremReductionToSubshift}}
\label{Sec:ProofOfTheoremReductionToSubshift}

This section is devoted to the proof of Theorem \ref{TheoremReductionToSubshift}. For $i=1,2,3$ consider the quadratic irrationals $\eta_i$ with $\phi_\infty<\eta_1<\eta_2<\eta_3$ defined by
\begin{eqnarray*}
&&
\eta_1:=
7\cdot\frac{[1,4,2,\overline{1,5}]+5+[1,5,1,\overline{1,5}]}{4}
=
11,655309\pm10^{-6}
\\
&&
\eta_2:=
7\cdot\frac{[1,\overline{1,6}]+6+[\overline{6,1}]}{4}=
11,688957\pm10^{-6}
\\
&&
\eta_3:=
7\cdot\big(
1+[1,\overline{1,6}]+[\overline{6,1}]
\big)
=
11,755835\pm10^{-6}.
\end{eqnarray*}

In \S~\ref{SecIntermediateAndSmallGraph} below we consider a graph $\cA$, obtained from a subgraph of $\cG$, that we call \emph{accessible graph}. Then we introduce a graph $\cI$ called \emph{intermediate graph}, whose vertices are the vertices of $\cA$ and whose arrows are some combinations of the arrows of $\cA$, and finally we define a subgraph $\cS$ of $\cI$, called \emph{small graph}. In Proposition \ref{PropReductionToIntermediateGraph} we show that data $X_j$ and $\alpha$ with $L(X_j,\alpha)<\eta_3$ satisfy
$
\support(X_j,\alpha)\subset\cA
$
and $a_n\leq6$ eventually, which means that the path $(X_j,\alpha)$ can be decomposed into elementary operations which are arrows of the intermediate graph $\cI$. Then in Proposition \ref{propReductionToSmallgraph} we show that 
if $L(X_j,\alpha)<\eta_1$ then the path $(X_j,\alpha)$ can be decomposed into elementary operations which are arrows of the small graph $\cS$, and this is the main technical tool in the proof of Theorem \ref{TheoremReductionToSubshift}. As an intermediate step towards the proof of Proposition \ref{propReductionToSmallgraph}, in Corollary~\ref{CoroReductionTointermediateGraph} we prove that if $L(X_j,\alpha)<\eta_2$ then $\support(X_j,\alpha)\subset\cA$ and $a_n\leq5$ eventually.

\subsection{The intermediate graph $\cI$ and the small graph $\cS$}

\label{SecIntermediateAndSmallGraph}

For $X_j\in\cO$ and a positive integer $a$, let $Y_k$ and $Z_l$ be the two elements in $\cO$ such that we have respectively $T^aR(X_j)=Y_k$ and $T^{-a}R(X_j)=Z_l$.  Represent the two elementary operations above as
$$
X_j
\quad\substack{a+\\ \longrightarrow}\quad
Y_k
\textrm{ and }
X_j
\quad\substack{a-\\ \longrightarrow}\quad
Z_l.
$$

The \emph{accessible graph} $\cA$  (shown  to the right in Figure \ref{FigureIntermediateGraphA}) is the graph whose vertices are $C_2$, $C_3$, $C_4$, $C_5$, $C_6$, $E_2$, $E_3$ and $F_0$ and whose elementary arrows are
\begin{eqnarray*}
&&
T:C_i\to C_{i+1}
\textrm{ for }
i=2,3,4,5
\\
&&
T^3:C_6\to C_2
\\
&&
T:F_0\to F_0
\\
&&
T:E_2\to E_3
\\
&&
R:E_2\to C_6
\textrm{ , }
R:E_3\to C_2
\textrm{ and }
R:E_3\to C_5.
\end{eqnarray*}

The accessible graph is  induced from a proper subgraph of $\cG$ as follows.  
 In Figure \ref{FigureIntermediateGraphA}, the accessible graph $\cA$ is represented on the right side, whereas on the left side one can see a graph which is a proper subgraph of the graph $\cG$: $\cA$ is obtained from the subgraph of $\cG$ by removing the vertices $C_0$ and $C_1$ and the arrows $T:C_6\to C_0$, $T:C_0\to C_1$ and $T:C_1\to C_2$, then replacing them by the arrow $T^3:C_6\to C_2$. 
We stress that even if the elements $C_0$ and $C_1$ are not part of the accessible graph $\cA$,  paths $(X_j,\alpha)$ whose support is in $\cA$ and which only use arrows in $\cA$ pass through these to elements, but never at times corresponding to Gauss approximations, since no $R$ arrow leaves from this vertices. 

\begin{figure}[!ht]
\includegraphics[width=.33\textwidth]{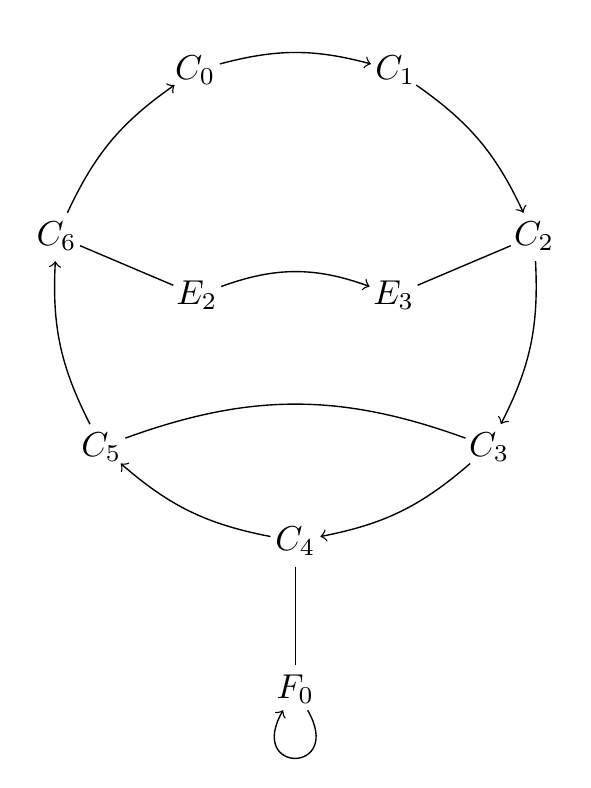}
\hspace{.6cm}
\includegraphics[width=.33\textwidth]{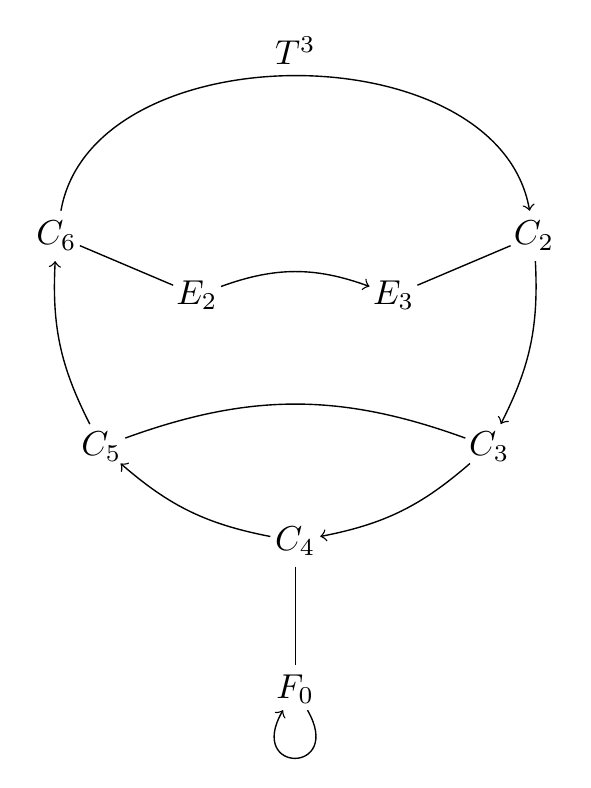}
\caption{Removing the vertices $C_0$ and $C_1$ and the arrows $T:C_6\to C_0$, $T:C_0\to C_1$ and $T:C_1\to C_2$ from the graph on the left of the picture, which is a subgraph of $\cG$, and replacing them by the arrow $T^3:C_6\to C_2$, one gets the accessible graph $\cA$, represented on the right of the picture, which contains all paths that can contribute to the bottom of the Lagrange spectrum.}
\label{FigureIntermediateGraphA}
\end{figure}

We say that an elementary operation
$
X_j
\quad\substack{a\pm\\ \longrightarrow}\quad
Y_k
$ 
is \emph{admissible in} $\cA$ if it appears infinitely many times in a path $(X_j,\alpha)$ with support in $\cA$.

Let us now introduce the \emph{intermediate graph} $\cI$, which is represented in Figure \ref{FigureIntermediateGraphB} and is the graph induced from $\cA$ taking some admissible elementary operations (see Lemma \ref{LemInducingIntermediateGraph} below). The vertices of the intermediate graph $\cI$ are $C_2$, $C_3$, $C_4$, $C_5$, $C_6$, $E_2$, $E_3$ and $F_0$. The arrows in $\cI$ correspond to the $30$ elementary operations listed here below, where for convenience of notation we also introduce a short name for each, which is either $\cF_{\xi}$ or $\cG_{\xi}$, or a primed version of the same names, i.e. $\cF_{\xi}'$ or $\cG_{\xi}'$, and where $\xi$ is a label. Primes reflect a symmetry of the graph under the involution $\psi$ defined below, primed arrows are shown in a lighter shade of the corresponding unprimed version in Figure \ref{FigureIntermediateGraphB}. The label $\xi$ associated to an elementary operation 
$
X_j
\quad\substack{a\pm\\ \longrightarrow}\quad
Y_k
$ 
is a contraction (preserving the injectivity of the labelling) of the three-symbols data $XaY$, representing the cusps $X$ and $Y$ respectively of the starting point $X_j$ and the ending point $Y_k$ of the elementary operation, and the number $a$ of iterates of $T$ or of $T^{-1}$ that it contains. The elementary operations named by $\cF_\xi$ are the ones which will survive in a further reduction.

\begin{eqnarray*}
&&
\cF_1:=
C_6\quad\substack{1+\\ \longrightarrow }\quad E_3
\quad
\textrm{ and }
\quad
\cF'_1:=
C_2\quad\substack{1-\\ \longrightarrow }\quad E_2
\\
&&
\cF_2:=
C_5\quad\substack{2+\\ \longrightarrow }\quad C_5
\quad
\textrm{ and }
\quad
\cF'_2:=
C_3\quad\substack{2-\\ \longrightarrow }\quad C_3
\\
&&
\cF_3:=
E_3\quad\substack{3-\\ \longrightarrow }\quad C_6
\quad
\textrm{ and }
\quad
\cF'_3:=
E_2\quad\substack{3+\\ \longrightarrow }\quad C_2
\\
&&
\cF_{C4}:=
C_5\quad\substack{4-\\ \longrightarrow }\quad C_6
\quad
\textrm{ and }
\quad
\cF'_{C4}:=
C_3\quad\substack{4+\\ \longrightarrow }\quad C_2
\\
&&
\cF_{E4}:=
E_3\quad\substack{4-\\ \longrightarrow }\quad C_5
\quad
\textrm{ and }
\quad
\cF'_{E4}:=
E_2\quad\substack{4+\\ \longrightarrow }\quad C_3
\end{eqnarray*}

\begin{eqnarray*}
&&
\cG_{CC}:=
C_5\quad\substack{1+\\ \longrightarrow }\quad C_4
\quad
\textrm{ and }
\quad
\cG'_{CC}:=
C_3\quad\substack{1-\\ \longrightarrow }\quad C_4
\\
&&
\cG_{CF}:=
C_4\quad\substack{1+\\ \longrightarrow }\quad F_0
\quad
\textrm{ and }
\quad
\cG'_{CF}:=
C_4\quad\substack{1-\\ \longrightarrow }\quad F_0
\\
&&
\cG_{FC}:=
F_0\quad\substack{1+\\ \longrightarrow }\quad C_5
\quad
\textrm{ and }
\quad
\cG'_{FC}:=
F_0\quad\substack{1-\\ \longrightarrow }\quad C_3
\\
&&
\cG_{E5}:=
E_3\quad\substack{5-\\ \longrightarrow }\quad C_4
\quad
\textrm{ and }
\quad
\cG'_{E5}:=
E_2\quad\substack{5+\\ \longrightarrow }\quad C_4
\\
&&
\cG_{C5}:=
C_5\quad\substack{5-\\ \longrightarrow }\quad C_5
\quad
\textrm{ and }
\quad
\cG'_{C5}:=
C_3\quad\substack{5+\\ \longrightarrow }\quad C_3
\\
&&
\cG_{5C}:=
F_0\quad\substack{5-\\ \longrightarrow }\quad C_6
\quad
\textrm{ and }
\quad
\cG'_{5C}:=
F_0\quad\substack{5+\\ \longrightarrow }\quad C_2
\\
&&
\cG_{6C}:=
C_5\quad\substack{6-\\ \longrightarrow }\quad C_4
\quad
\textrm{ and }
\quad
\cG'_{6C}:=
C_3\quad\substack{6+\\ \longrightarrow }\quad C_4
\\
&&
\cG_{F6}:=
F_0\quad\substack{6-\\ \longrightarrow }\quad C_5
\quad
\textrm{ and }
\quad
\cG'_{6C}:=
F_0\quad\substack{6+\\ \longrightarrow }\quad C_3
\\
&&
\cG_{E6}:=
E_3\quad\substack{6-\\ \longrightarrow }\quad C_3
\quad
\textrm{ and }
\quad
\cG'_{E6}:=
E_2\quad\substack{6+\\ \longrightarrow }\quad C_5
\\
&&
\cG_{C6}:=
C_5\quad\substack{6+\\ \longrightarrow }\quad C_2
\quad
\textrm{ and }
\quad
\cG'_{C6}:=
C_3\quad\substack{6-\\ \longrightarrow }\quad C_6.
\end{eqnarray*}

\begin{lemma}
\label{LemInducingIntermediateGraph}
The elementary operations 
$
X_j
\quad\substack{a\pm\\ \longrightarrow}\quad
Y_k
$ 
with $1\leq a\leq 6$ which are admissible in $\cA$ are exactly the 30 elementary arrows of the intermediate graph $\cI$.
\end{lemma}

\begin{proof} 
The proof is a direct verification. For any vertex $X_j$ of the accessible graph $\cA$ we consider all the admissible elementary operations as in the statement which start at $X_j$. For example 
$
C_6
\quad\substack{1+\\ \longrightarrow}\quad
E_3
$ 
is the only elementary operation starting at $C_6$. Then compatibility with concatenation implies that 
$
E_3
\quad\substack{a-\\ \longrightarrow}\quad
C_{9-a}
$ 
with $a=3,4,5,6$ are the only admissible operations starting at $E_3$, where $a=1,2$ are not allowed because $C_0$ and $C_1$ do not belong to $\cA$. Details are left to the reader.
\end{proof}

\begin{figure}[!ht]
\includegraphics[width=.48\textwidth]{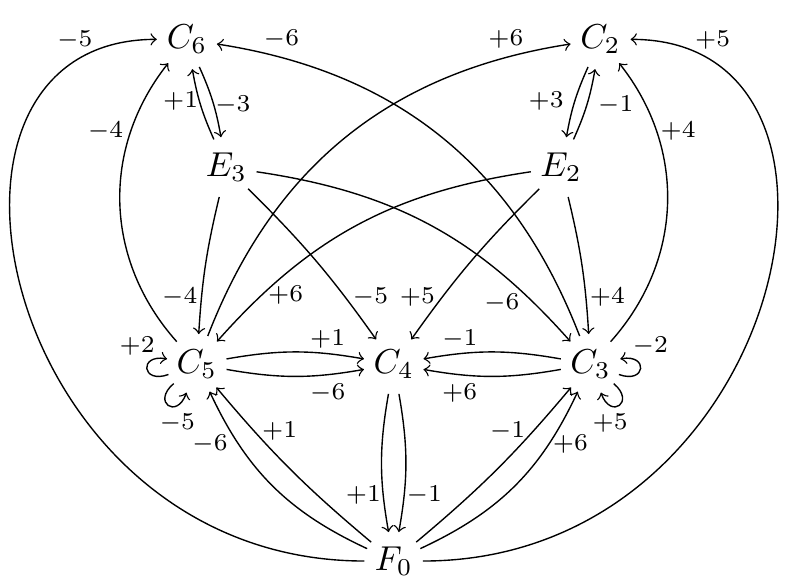}
\hspace{2mm}
\includegraphics[width=.48\textwidth]{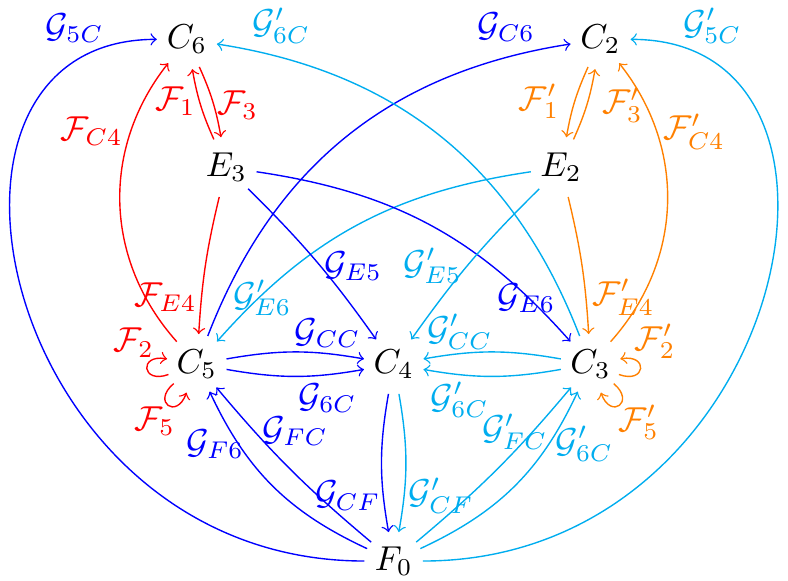}
\caption{The elementary operations with are admissible in the accessible graph $\cA$ are the arrows of the intermediate graph $\cI$. In the left copy, these elementary operations are represented explicitly, whereas in the right copy they are represented by their short names. Colors are explained in the text.}
\label{FigureIntermediateGraphB}
\end{figure}


The graph $\cI$ has an internal symmetry which we will exploit to further reduce the paths to consider to study the bottom of the spectrum. In Figure \ref{FigureIntermediateGraphB} this symmetry correspond to a reflection through a vertical axes passing through $C_4$ and $F_0$. Formally, we describe this symmetry by defining an involution $\psi:\cI\to\cI$ which acts on vertices by sending
$$
\psi(C_2)=C_6
\textrm{ , }
\psi(C_3)=C_5
\textrm{ , }
\psi(E_3)=E_2
\textrm{ , }
\psi(C_4)=C_4
\textrm{ , }
\psi(F_0)=F_0 .
$$
One can easily verify that  $\psi^2=\textrm{Id}$ and hence $\psi$ is indeed an involution.  Extend the function $\psi$ to the set of paths on $\cS$ by
\begin{eqnarray*}
&&
\psi\big(\cF_\xi\big):=\cF'_\xi
\textrm{ for any label }
\xi=1,2,3,C4,E4
\\
&&
\psi\big(\cG_\xi\big):=\cG'_\xi
\textrm{ for any label }
\xi=CC,CF,FC,E5,C5,5C,6C,F6,E6,C6.
\end{eqnarray*}

\begin{remark}
\label{involutionrespectsmultiplicities}
Observe that the multiplicity is preserved by $\psi$, that is
$
m\big(\psi(X_j)\big)=m(X_j)
$
for any element $X_j\in\cI$. Furthermore, notice that $\psi$ maps arrows labeled by $+a$ to $-a$ and conversely. 
\end{remark}

If we consider only paths which eventually do not cross the vertex $F_0$ (and hence neither the vertex $C_4$), one can check that the remaining paths live in one of two connected components, mapped to each other by the involution. Let us hence define a graph which encodes (a subset of) the operations of one of these two connected components. 
 
The \emph{small graph} (show in Figure \ref{smallgraph}) is  the subgraph $\cS$ of $\cI$  whose vertices are the elements $C_5$, $C_6$, $E_3$ and whose arrows are $\cF_1$, $\cF_2, \cF_3, \cF_{4C}, \cF_{4E}$.  
Observe that the subgraph $\cI'$ of $\cI$ whose vertices are $C_2$, $C_3$, $C_5$, $C_6$, $E_2$, $E_3$ and whose arrows are $\cF_\xi$ and $\cF'_\xi$ for $\xi=1,2,3,4C,4E$ has two connected components, one being the small graph $\cS$ and the other being the image $\psi(\cS)$ of $\cS$ under the involution $\psi$.

\begin{figure}[!ht]
\includegraphics[width=.25\textwidth]{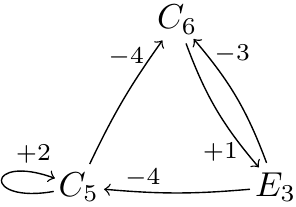}
\caption{The \emph{small graph} $\cS$.}
\label{smallgraph}
\end{figure}

Proposition \ref{propReductionToSmallgraph} below is the main technical result in the proof of Theorem \ref{TheoremReductionToSubshift}. The proof of Proposition \ref{propReductionToSmallgraph} is the subject of the next subsections \S \ref{SecReductionTo IntermediateGraphs},  \S \ref{Secblocks} and \S \ref{SecProofOfPropositionReductionToSmallGraph}.

\begin{proposition}
\label{propReductionToSmallgraph}
Consider $X_j\in\cO$ and $\alpha$ such that
$$
L(X_j,\alpha)<\eta_1.
$$
Then, eventually, the only elementary operations appearing in the path $(X_j,\alpha)$ are:
\begin{enumerate}
\item
either arrows of the small graph $\cS$, that belong to  the list: $\cF_1$, $\cF_2$, $\cF_3$, $\cF_{4C}$, $\cF_{4E}$;
\item
or arrows of the image $\psi(\cS)$ of the small graph, that is belong to the list: $\cF'_1$, $\cF'_2$, $\cF'_3$, $\cF'_{4C}$, $\cF'_{4E}$.
\end{enumerate}
\end{proposition}

Lemma \ref{LemmaMaxMinCantorset} below is a classical result on continued fraction playing a relevant role in several parts of this section. The proof is left to the reader.

\begin{lemma}
\label{LemmaMaxMinCantorset}
Fix a positive integer $M$. Consider $m\in\NN$ and positive integers $b_1,\dots,b_m$. Let $K$ be the Cantor set of those real numbers $x$ of the form
$$
x=
[b_1,\dots,b_m,a_1,a_2,\dots]
\textrm{ where }
1\leq a_n\leq M
\textrm{ for any }
n,
$$
that is the set of those $x$ whose continued fraction expansion starts with the prescribed entries $b_1,\dots,b_m$ and then all the other entries $a_n$ satisfy $1\leq a_n\leq M$. Then we have
\begin{eqnarray*}
&&
\min K=[b_1,\dots,b_m,\overline{M,1}]
\textrm{ and }
\max K=[b_1,\dots,b_m,\overline{1,M}]
\textrm{ if }
m
\textrm{ is even }
\\
&&
\min K=[b_1,\dots,b_m,\overline{1,M}]
\textrm{ and }
\max K=[b_1,\dots,b_m,\overline{M,1}]
\textrm{ if }
m
\textrm{ is odd. }
\\
\end{eqnarray*}
In particular, for the degenerate case $m=0$, when there is no prescribed beginning for the continued fraction expansion of the elements of $K$, we have
$$
[\overline{M,1}]=\min K
\textrm{ and }
[\overline{1,M}]=\max K.
$$
\end{lemma}

\subsection{Reduction to the intermediate graph $\cI$}
\label{SecReductionTo IntermediateGraphs}

Observe that
$$
\eta_3
<
7\cdot\frac{7+2[\overline{7,1}]}{4}
=12,693741\pm10^{-6}.
$$

\begin{lemma}
\label{lemCutEntriesAbove6}
Consider $X_j\in\cO$ and $\alpha=[a_1,a_2,\dots]$ irrational such that
$$
L(X_j,\alpha)<
7\cdot\frac{7+2[\overline{7,1}]}{4}.
$$
Then we eventually have $1\leq a_n\leq 6$.
\end{lemma}

\begin{proof}
Fix any $\delta>0$. If the continued expansion of $\alpha$ contains infinitely many $a_n\geq 8$, then for any such $n$ which is big enough and for $i=a_n$ we have
$
D(n,i,\alpha)=D(n,a_n,\alpha)>8
$
and
$$
L(X_j,\alpha)+\delta\geq
7\cdot\frac{D(n,a_n,\alpha)}
{m^2\big(R\cdot g(a_1,\dots,a_{n})\cdot X_j\big)}
\geq
7\cdot\frac{D(n,a_n,\alpha)}{4}
>14.
$$
Since $\delta$ is arbitrary, we get $L(X_j,\alpha)\geq 14$, which is absurd. Now assume that the continued expansion of $\alpha$ contains infinitely many $a_n\geq 7$. For any such $n$ which is big enough, and for $i=a_n$, Lemma \ref{LemmaMaxMinCantorset} implies
$
D(n,a_n,\alpha)> 7+2[\overline{7,1}]
$
and thus
$$
L(X_j,\alpha)+\delta\geq
\frac{D(n,a_n,\alpha)}
{m^2\big(R\cdot g(a_1,\dots,a_{n})\cdot X_j\big)}\geq
7\cdot\frac{7+2[\overline{7,1}]}{4}.
$$
Since $\delta$ is arbitrary, it follows
$
L(X_j,\alpha)\geq7(7+2[\overline{7,1}])/4
$,
which is absurd.
\end{proof}

\begin{lemma}
\label{lemObstructionForMultiplicity1}
Consider $X_j\in\cO$ and $\alpha=[a_1,a_2,\dots]$ irrational such that $L(X_j,\alpha)<\eta_3$. Assume that there exists $Y_k\in \support(X_j,\alpha)$ with $m^2(R\cdot Y_k)=1$. Then for any $n$ with
$
g(a_1,\dots,a_n)\cdot X_j=Y_k
$
we eventually have
\begin{eqnarray*}
&&
2\leq a_n\leq 6
\textrm{ , }
a_{n+1}=1
\textrm{ , }
2\leq a_{n+2}\leq 6
\\
&&
m^2\big(
R g(a_1,\dots,a_{n-1})\cdot X_j
\big)
=
m^2\big(
R g(a_1,\dots,a_{n+1})\cdot X_j
\big)
=4.
\end{eqnarray*}
\end{lemma}

\begin{proof}
If $n$ is as in the statement, since
$
m^2\big(Rg(a_1,\dots,a_n)\cdot X_j\big)=1
$,
then we have $a_{n+1}=1$, otherwise Formula \eqref{eqFormulaLagrangeConstant} would give
$$
L(X_j,\alpha)+\delta\geq
7\cdot\frac{D(n,a_n,\alpha)}
{m^2\big(R\cdot g(a_1,\dots,a_{n})\cdot X_j\big)}>
7a_{n+1}\geq14,
$$
where $\delta>0$ is an arbitrarily small positive constant, which is absurd. According to Lemma \ref{lemCutEntriesAbove6} we have $1\leq a_n\leq 6$. Moreover, if either $a_n=1$ or $a_{n+2}=1$ then Lemma \ref{LemmaMaxMinCantorset} implies
$$
L(X_j,\alpha)+\delta\geq
7\cdot
\frac{D(n,a_n,\alpha)}
{m^2\big(R\cdot g(a_1,\dots,a_{n})\cdot X_j\big)}>
7D(n,1,\alpha)\geq
7\big(1+[1,\overline{1,6}]+[\overline{6,1}]\big)
$$
for any arbitrarily small positive constant $\delta$ and any $n$ big enough, which is absurd. Since $a_n\geq 2$ and $a_{n+2}\geq 2$, it follows
$
m^2\big(
R g(a_1,\dots,a_{n-1})\cdot X_j
\big)
=
m^2\big(
R g(a_1,\dots,a_{n+1})\cdot X_j
\big)
=4
$,
otherwise, arguing as in the beginning of the proof one gets $L(X_j,\alpha)\geq 14$, which is absurd.
\end{proof}

\begin{proposition}
\label{PropReductionToIntermediateGraph}
Consider $X_j\in\cO$ and $\alpha=[a_1,a_2,\dots]$ irrational such that
$$
L(X_j,\alpha)<\eta_3.
$$
Then $\support(X_j,\alpha)$ is contained in the accessible graph $\cA$,  we eventually have $1\leq a_n\leq 6$ and the path $(X_j,\alpha)$ is eventually decomposed into elementary operations which are arrows of the intermediate graph $\cI$. 
\end{proposition}

\begin{proof}
Let $Y_k\in\support(X_j,\alpha)$ with $m^2(R\cdot Y_k)=1$. Consider $n$ with
$
g(a_1,\dots,a_n)\cdot X_j=Y_k
$.
Assume that $n$ is even. According to Lemma \ref{lemObstructionForMultiplicity1} we have $a_{n+1}=1$ and therefore
$$
TR\cdot Y_k=g(a_1,\dots,a_{n+1})\cdot X_j.
$$
The last equality, and again Lemma \ref{lemObstructionForMultiplicity1} imply $m^2(RTR\cdot Y_k)=4$, which has for solutions
$$
TR\cdot Y_k=A_5,B_4,C_3,C_5,E_2,E_3,F_0,
$$
corresponding respectively to
$
Y_k=E_4,A_6,E_3,F_0,B_5,C_6,C_4
$.
In such list, the elements $Y_k$ satisfying $m^2(R\cdot Y_k)=1$ are
$$
Y_k=E_4,A_6,B_5,C_6,C_4.
$$
Moreover, we also have the identity
$$
T^{-a_n}Rg(a_1,\dots,a_{n-1})\cdot X_j=Y_k,
$$
thus Lemma \ref{lemObstructionForMultiplicity1} implies also $m^2(T^{a_n}\cdot Y_k)=4$. Independently on $a_n$, the only elements in the list $E_4,A_6,B_5,C_6,C_4$ satisfying this last condition are $Y_k=C_4$ and $Y_k=C_6$, which are both elements of the accessible graph $\cA$. Now assume that $n$ is odd, so that we have the identity
$$
T^{-1}R\cdot Y_k=g(a_1,\dots,a_{n+1})\cdot X_j.
$$
Lemma \ref{lemObstructionForMultiplicity1} and the identity above imply $m^2(RT^{-1}R\cdot Y_k)=4$, which has for solutions
$$
Y_k=B_3,E_1,F_0,E_2,C_2,A_4,C_4.
$$
In such list, the elements $Y_k$ satisfying $m^2(R\cdot Y_k)=1$ are $
Y_k=B_3,E_1,C_2,A_4,C_4
$.
Moreover, we also have the identity
$$
T^{a_n}Rg(a_1,\dots,a_{n-1})\cdot X_j=Y_k,
$$
thus Lemma \ref{lemObstructionForMultiplicity1} implies also
and $m^2(T^{-a_n}\cdot Y_k)=4$. Independently on $a_n$, the only elements in the list $B_3,E_1,C_2,A_4,C_4$ satisfying this last condition are $Y_k=C_4$ and $Y_k=C_2$, which are both elements of the accessible graph $\cA$.

All other elements of $\cO$ satisfy $m^2(R\cdot Y_k)=4$ and they are $
Y_k=A_5,B_4,C_3,C_5,E_2,E_3,F_0
$.
In this list, the only elements which are not in the accessible graph $\cA$ are $Y_k=B_4$ and $Y_k=A_5$. If
$
B_4\in\support(X_j,\alpha)
$
then there exist $i$ with $0\leq i\leq 6$ with
$
R\cdot B_i\in\support(X_j,\alpha)
$.
Since
$
m^2(R\cdot R\cdot B_i)=m^2(B_i)=1
$
for any $i=0,\dots,6$, Lemma \ref{lemObstructionForMultiplicity1} (the part of the statement corresponding to $a_{n+1}=1$) implies
\begin{eqnarray*}
&&
\textrm{ either }
E_1=RT\cdot B_4=R\cdot B_5\in\support(X_j,\alpha)
\\
&&
\textrm{ or }
A_6=RT^{-1}\cdot B_4=R\cdot B_3\in\support(X_j,\alpha).
\end{eqnarray*}
The first case corresponds to $B_4=T^{-1}R\cdot E_1$ and Lemma \ref{lemObstructionForMultiplicity1} again implies  $E_1=T^{a}R\cdot Z_l$ for some $a$ with $2\leq a\leq 6$ and some $Z_l\in\cO$ with $m^2(R\cdot Z_l)=4$, and this is absurd since $m^2(T^{-a}\cdot E_1)=1$ for any $a$. The second case corresponds to $B_4=TR\cdot A_6$ and reasoning similarly we get necessarily $A_6=T^{-a}R\cdot Z_l$ for some $a$ with $2\leq a\leq 6$ and some $Z_l\in\cO$ with $m^2(R\cdot Z_l)=4$, which is absurd because $m^2(T^{a}\cdot A_6)=1$ for any $a$. If
$
A_5\in\support(X_j,\alpha)
$,
arguing as for $B_4$, we get
\begin{eqnarray*}
&&
\textrm{ either }
B_3=RT\cdot A_5\in\support(X_j,\alpha)
\\
&&
\textrm{ or }
E_4=RT^{-1}\cdot A_5\in\support(X_j,\alpha).
\end{eqnarray*}
The first case corresponds to $A_5=T^{-1}R\cdot B_3$ and implies $B_3=T^{a}R\cdot Z_l$ for some $Z_l\in\cO$ with $m^2(R\cdot Z_l)=4$, and this is impossible since $m^2(T^{-a}B_3)=1$ for any $a$. The second case corresponds to $A_5=TR\cdot E_4$ and implies $E_4=T^{-a}R\cdot Z_l$ for some $Z_l\in\cO$ with $m^2(R\cdot Z_l)=4$, and this is impossible since $m^2(T^aE_4)=1$ for any $a$. We proved that $\support(X_j,\alpha)$ is contained in the accessible graph $\cA$. Finally, according to Lemma \ref{lemCutEntriesAbove6} we eventually have also $1\leq a_n\leq 6$. It follows, according to Lemma \ref{LemInducingIntermediateGraph},  that the path $(X_j,\alpha)$ is eventually decomposed into elementary operations which are arrows of the intermediate graph $\cI$.  The Proposition is proved.
\end{proof}

\subsection{Blocks}\label{Secblocks}

A \emph{block} $H$ is a finite segment of a path $(X_j,\alpha)$ generated by some element $X_j\in\cO$ and some $\alpha$ irrational. More precisely a block is a triple of data
$
H=(Y_k,\{b_1,\dots,b_m\},\epsilon)
$
where $Y_k\in\cO$, the entries $b_1,\dots,b_m$ are positive integers and $\epsilon\in\{+,-\}$. If $H$ is a block, there exists $X_j\in\cO$ and $\alpha=[a_1,a_2,\dots]$ such that $Y_j\in\support(X_j,\alpha)$, that is
$
Y_k=g(a_1,\dots,a_n)\cdot X_j
$
for infinitely many $n$, and moreover for any such $n$ we have
$
a_{n+1}=b_1,\dots,a_{n+m}=b_m
$
and $\epsilon=+$ if $n$ is odd and $\epsilon=-$ if $n$ is even. In the following we will use often an alternative representation of blocks, with the explicit description of all operations as above. For example the blocks $H_1=(C_5,[5,6,1],-)$ and $H_2=(E_2,[3,1,6,5],+)$ will be represented respectively by the two sequences of operations
\begin{eqnarray*}
&&
H_1=
C_5
\quad\substack{5-\\ \longrightarrow}\quad
C_5
\quad\substack{6+\\ \longrightarrow}\quad
C_2
\quad\substack{1-\\ \longrightarrow}\quad
E_2
\\
&&
H_2=
E_2
\quad\substack{3+\\ \longrightarrow}\quad
C_2
\quad\substack{1-\\ \longrightarrow}\quad
E_2
\quad\substack{6+\\ \longrightarrow}\quad
C_5
\quad\substack{5-\\ \longrightarrow}\quad
C_4
\end{eqnarray*}
which appear as finite segments of the path
$
(E_2,\alpha=[\overline{3,1,6,5,6,1}])
$. 
We will assume that blocks are contained in the accessible graph $\cA$. 

\medskip

Let 
$
H=(Y_k,\{b_1,\dots,b_m\},\epsilon)
$ 
be a block which is contained in the accessible graph $\cA$. We define
$$
L^{inf}(H):=\inf L(X_j,\alpha),
$$
where $(X_j,\alpha)$ varies among all the paths with $\support(X_j,\alpha)\subset\cA$ which contains the block $H$ infinitely many times. If $M$ is a positive integer such that $b_i\leq M$ for all the entries of the block $H$ we define also
$
L_M^{inf}(H):=\inf L(X_j,\alpha)
$
where $(X_j,\alpha)$ varies among all the paths with $\support(X_j,\alpha)\subset\cA$ which contains the block $H$ infinitely many times and such that $\alpha=[a_1,a_2,\dots]$ eventually satisfies $a_n\leq M$. We have the following simple but very useful Lemma.

\begin{lemma}
\label{LemMinorationblocks}
Fix a positive integer $M$ and consider a block
$
H=(Y_k,[b_1,\dots,b_m],\epsilon)
$
of length $m\geq1$ such that $b_i\leq M$ for any $i$ with $1\leq i\leq m$. For any such $i$ we have
$$
L^{inf}_M(H)\geq
7\cdot
\frac
{[b_i,\dots,b_1,\overline{c,d}]+
b_{i+1}+
[b_{i+2},\dots,b_m,\overline{e,f}]}
{m^2(R\cdot g(b_1,\dots,b_i)\cdot Y_k)},
$$
where
\begin{eqnarray*}
&&
(c,d)=(e,f)=(M,1)
\textrm{ for }
i
\textrm{ even, }
m
\textrm{ odd }
\\
&&
(c,d)=(e,f)=(1,M)
\textrm{ for }
i,m
\textrm{ odd }
\\
&&
(c,d)=(M,1)
\textrm{ and }
(e,f)=(1,M)
\textrm{ for }
i,m
\textrm{ even }
\\
&&
(c,d)=(1,M)
\textrm{ and }
(e,f)=(1,M)
\textrm{ for }
i
\textrm{ odd, }
m
\textrm{ even. }
\end{eqnarray*}
\end{lemma}

\begin{proof}
Let $(X_j,\alpha)$, with $\alpha=[a_1,a_2,\dots]$ be any path containing the block $H$ infinitely many times and such that eventually $a_n\leq M$. Fix any $\delta>0$. Fix $i$ with $0\leq i\leq m-1$. Formula \eqref{eqFormulaLagrangeConstant} gives
\begin{eqnarray*}
&&
L(X_j,\alpha)=
7\limsup_{n\to\infty}
\max_{1\leq i\leq a_n}
\frac{D(n,i,\alpha)}
{m^2\big(R\cdot g(a_1,\dots,a_{n-1},i)\cdot X_i\big)}\geq
\\
&&
7\limsup_{n\to\infty}
\frac{D(n,a_n,\alpha)}
{m^2\big(R\cdot g(a_1,\dots,a_n)\cdot X_i\big)}
=
7\limsup_{n\to\infty}
\frac{[a_n,a_{n-1},\dots]+a_{n+1}+[a_{n+2},a_{n+3},\dots]}
{m^2\big(R\cdot g(a_1,\dots,a_n)\cdot X_i\big)}.
\end{eqnarray*}
The Lemma follows from Lemma \ref{LemmaMaxMinCantorset} observing that, since $H$ appears infinitely often in the path generated by $(X_j,\alpha)$, then there is $n$ arbitrarily big with
$$
a_{n+1-i}=b_1,
\dots,
a_{n+1}=b_{i+1},
\dots,
a_{n+m-i}=b_m.
$$
\end{proof}

The following Lemma is also practical to reduce the number of cases in our analysis. 

\begin{lemma}
\label{LemmaInvarianceBlocksInvolution}
Fix a positive integer $M$ and consider a block
$
H=(Y_k,[b_1,\dots,b_m],\epsilon)
$
of length $m\geq1$ such that $b_i\leq M$ for any $i$ with $1\leq i\leq m$. If $H$ is contained into the accessible graph $\cA$, then we have
$$
L^{inf}_M\big(\psi(H)\big)
=
L^{inf}_M\big(H\big).
$$
\end{lemma}

\begin{proof}
Consider a path $(X_j,\alpha)$ with $\support(X_j,\alpha)\subset\cA$, containing the block $H$ infinitely many times and such that $\alpha=[a_1,a_2,\dots]$ eventually satisfies $a_n\leq M$. We can assume without loosing in generality that all elementary operations composing $(X_j,\alpha)$ are arrows of the intermediate graph. Thus one can define a new path $(X_j',\alpha)$ by concatenation, in the same order, of the image under $\phi$ of the elementary operations composing $(X_j,\alpha)$, where obviously $\alpha=[a_1,a_2,\dots]$ keeps unchanged. The Lemma follows applying Formula \eqref{eqFormulaLagrangeConstant} and recalling Remark \ref{involutionrespectsmultiplicities}.
\end{proof}

\begin{lemma}
\label{lemCutEntriesAbove5}
The inequalities below hold.
\begin{eqnarray*}
&&
L^{inf}_6
\big(
F_0\quad\substack{6+\\ \longrightarrow }\quad C_3
\big)
=
L^{inf}_6
\big(
F_0\quad\substack{6-\\ \longrightarrow }\quad C_5
\big)
\geq\eta_2.
\\
&&
L^{inf}_6
\big(
C_3\quad\substack{6+\\ \longrightarrow }\quad C_4
\big)
=
L^{inf}_6
\big(
C_5\quad\substack{6-\\ \longrightarrow }\quad C_4
\big)
\geq\eta_2.
\\
&&
L^{inf}_6
\big(
C_5\quad\substack{6+\\ \longrightarrow }\quad C_2
\big)
=
L^{inf}_6
\big(
C_3\quad\substack{6-\\ \longrightarrow }\quad C_6
\big)
\geq\eta_2.
\\
&&
L^{inf}_6
\big(
E_2\quad\substack{6+\\ \longrightarrow }\quad C_5
\big)
=
L^{inf}_6
\big(
E_3\quad\substack{6-\\ \longrightarrow }\quad C_3
\big)
\geq\eta_2.
\\
\end{eqnarray*}
\end{lemma}

\begin{proof}
The four equalities for the eight blocks in the statement follow from Lemma \ref{LemmaInvarianceBlocksInvolution} and from the correspondence between these blocks under the involution $\psi$, that is
\begin{eqnarray*}
&&
F_0\quad\substack{6-\\ \longrightarrow }\quad C_5=
\psi
\big(
F_0\quad\substack{6+\\ \longrightarrow }\quad C_3
\big)
\\
&&
C_5\quad\substack{6-\\ \longrightarrow }\quad C_4=
\psi
\big(
C_3\quad\substack{6+\\ \longrightarrow }\quad C_4
\big)
\\
&&
C_3\quad\substack{6-\\ \longrightarrow }\quad C_6=
\psi
\big(
C_5\quad\substack{6+\\ \longrightarrow }\quad C_2
\big)
\\
&&
E_3\quad\substack{6-\\ \longrightarrow }\quad C_3=
\psi
\big(
E_2\quad\substack{6+\\ \longrightarrow }\quad C_5
\big).
\end{eqnarray*}
and  by Remark \ref{involutionrespectsmultiplicities}, that implies that the desired  values of $L_6^{inf}$ are the same. 

The lemma thus follows proving that $L^{inf}_6(H_i)\geq\eta_2$ for $i=1,\dots,4$, where we set for simplicity
$$
H_1:=
F_0\quad\substack{6+\\ \longrightarrow }\quad C_3
\textrm{ , }
H_2:=
C_3\quad\substack{6+\\ \longrightarrow }\quad C_4
\textrm{ , }
H_3:=
C_5\quad\substack{6+\\ \longrightarrow }\quad C_2
\textrm{ , }
H_4:=
E_2\quad\substack{6+\\ \longrightarrow }\quad C_5.
$$
The block $H_1$ has as prolongation
$$
H'_1:=
C_4\quad\substack{1-\\ \longrightarrow }\quad
F_0\quad\substack{6+\\ \longrightarrow }\quad
C_3,
$$
Lemma \ref{LemMinorationblocks} with parameters $m=2$, $i=2$ implies
$$
L^{inf}(H'_1)
\geq
7\cdot
\frac{[1,\overline{1,6}]+6+[\overline{6,1}]}{4}=\eta_2.
$$
Since $m(C_4)=1$, then
$
L^{inf}
\big(
C_4\quad\substack{a-\\ \longrightarrow }\quad
F_0
\big)
>7\cdot2
$
for any $a\geq 2$ and therefore
$$
L^{inf}(H_1)\geq\min\{14,L^{inf}(H'_1)\}\geq\eta_2.
$$
The block $H_2$ has as prolongation the block
$$
H'_2:=
C_3\quad\substack{6+\\ \longrightarrow }\quad
C_4\quad\substack{1-\\ \longrightarrow }\quad
F_0,
$$
Lemma \ref{LemMinorationblocks} with parameters $m=2$, $i=1$ implies
$$
L^{inf}(H'_2)
\geq
7\cdot
\frac{[1,\overline{1,6}]+6+[\overline{6,1}]}{4}=\eta_2.
$$
Since $m(C_4)=1$, then
$
L^{inf}
\big(
C_4\quad\substack{a-\\ \longrightarrow }\quad
F_0
\big)
>7\cdot2
$
for any $a\geq 2$ and therefore
$$
L^{inf}(H_2)\geq\min\{14,L^{inf}(H'_2)\}\geq\eta_2.
$$
The block $H_3$ has as prolongation the block
$$
H'_3:=
C_5\quad\substack{6+\\ \longrightarrow }\quad
C_2\quad\substack{1-\\ \longrightarrow }\quad
E_3,
$$
Lemma \ref{LemMinorationblocks} with parameters $m=2$, $i=2$ implies
$$
L^{inf}(H'_3)
\geq
7\cdot
\frac{[1,\overline{1,6}]+6+[\overline{6,1}]}{4}=\eta_2.
$$
Since $m(C_2)=1$, then
$
L^{inf}
\big(
C_2\quad\substack{a-\\ \longrightarrow }\quad
X_j
\big)
>7\cdot2
$
for any $a\geq 2$ and therefore
$$
L^{inf}(H_3)\geq\min\{14,L^{inf}(H')\}\geq\eta_2.
$$
In fact one can also observe that
$
C_2\quad\substack{1-\\ \longrightarrow }\quad E_3
$
is the only arrow of the intermediate graph $\cI$ which starts at $C_2$. Finally, the block $H_4$ has as prolongation the block
$$
H'_4:=
C_5\quad\substack{6+\\ \longrightarrow }\quad
C_2\quad\substack{1-\\ \longrightarrow }\quad
E_3,
$$
Lemma \ref{LemMinorationblocks} with parameters $m=2$, $i=2$ implies
$$
L^{inf}(H'_4)
\geq
7\cdot
\frac{[1,\overline{1,6}]+6+[\overline{6,1}]}{4}=\eta_2.
$$
Since $m(C_2)=1$, then
$
L^{inf}
\big(
C_2\quad\substack{a-\\ \longrightarrow }\quad
X_j
\big)
>7\cdot2
$
for any $a\geq 2$ and therefore
$$
L^{inf}(H_4)\geq\min\{14,L^{inf}(H'_4)\}\geq\eta_2.
$$
As for $H_3$, one can also observe that
$
C_2\quad\substack{1-\\ \longrightarrow }\quad E_3
$
is the only arrow of the intermediate graph which starts at $C_2$.
\end{proof}

\begin{corollary}
\label{CoroReductionTointermediateGraph}
Let $(X_j,\alpha)$ be pair with $X_j\in\cO$ and $\alpha=[a_1,a_2,\dots]$ is irrational such that
$$
L(X_j,\alpha)<\eta_2.
$$
Then $\support(X_j,\alpha)$ is contained in the accessible graph $\cA$ and we eventually have $1\leq a_n\leq 5$. 
\end{corollary}

\begin{proof}
Let $X_j\in\cO$ and $\alpha=[a_1,a_2,\dots]$ be data as in the statement. Proposition \ref{PropReductionToIntermediateGraph} implies $\support(X_j,\alpha)\subset\cA$ and $a_n\leq 6$ eventually. Therefore, for any block $H$ appearing infinitely many times in the path $(X_j,\alpha)$ we have $L^{inf}(H)=L^{inf}_6(H)$. According to Lemma \ref{lemCutEntriesAbove5}, none of the eight blocks listed in the Lemma appear infinitely many times in the path $(X_j,\alpha)$. Therefore, if there exists infinitely many $n$ with $a_n=6$ then either the block
$
C_4\quad\substack{6+\\ \longrightarrow }\quad F_0
$
or the block
$
C_4\quad\substack{6-\\ \longrightarrow }\quad F_0
$
appear infinitely many times in $(X_j,\alpha)$, but this is absurd because $(R\cdot C_4)=1$ and thus
$
L^{inf}
\big(
C_4\quad\substack{6+\\ \longrightarrow }\quad F_0
\big)>7\cdot6
$
and
$
L^{inf}
\big(
C_4\quad\substack{6-\\ \longrightarrow }\quad F_0
\big)>7\cdot6
$.
\end{proof}

\subsection{Proof of Proposition \ref{propReductionToSmallgraph}}
\label{SecProofOfPropositionReductionToSmallGraph}

Consider data $X_j\in\cO$ and $\alpha=[a_1,a_2,\dots]$ such that
$$
L(X_j,\alpha)<
\eta_1=
7\cdot
\frac{[1,4,2,\overline{1,5}]+5+[1,5,1,\overline{1,5}]}{4}
=
11,655309\pm10^{-6}.
$$
Since $\eta_1<\eta_2$, Corollary \ref{CoroReductionTointermediateGraph} implies that $\support(X_j,\alpha)$ is contained in the accessible graph $\cA$ and $a_n\leq 5$ eventually. The only elementary operations ending in $F_0$ are
$
C_4\quad\substack{a+\\ \longrightarrow }\quad F_0
$
and
$
C_4\quad\substack{a-\\ \longrightarrow }\quad F_0
$
with $a=1,\dots,5$. Since $m^2(R\cdot C_4)=1$, among these the only compatible with the assumption corresponds to $a=1$, that is the arrows $\cG_{CF}$ and $\cG'_{CF}$ of the intermediate graph $\cI$. We will now show that such elementary operations cannot be repeated infinitely many times in paths $(X_j,\alpha)$ with $L(X_j,\alpha)<\eta_1$.

\begin{lemma}
\label{LemCutFzero}
We have
$$
L^{inf}_5
\big(
C_4\quad\substack{1+\\ \longrightarrow }\quad F_0
\big)=
L^{inf}_5
\big(
C_4\quad\substack{1-\\ \longrightarrow }\quad F_0
\big)
\geq\eta_1.
$$
\end{lemma}

\begin{proof}
The equality in the statement holds because the two blocks correspond under the involution $\psi$, that is
$$
C_4\quad\substack{1-\\ \longrightarrow }\quad F_0
=
\psi
\big(
C_4\quad\substack{1+\\ \longrightarrow }\quad F_0
\big).
$$
The only arrow of the intermediate graph $\cI$ which starts at $C_2$ is
$
C_2\quad\substack{1-\\ \longrightarrow }\quad E_2
$,
therefore the elementary operation 
$
F_0\quad\substack{2-\\ \longrightarrow }\quad C_2
$
is not admissible on the accessible graph $\cA$. Moreover
$
C_2\quad\substack{1-\\ \longrightarrow }\quad E_2
$
is also the only arrow of $\cI$ ending in $E_2$, therefore the elementary operation 
$
E_2\quad\substack{2-\\ \longrightarrow }\quad C_4
$
is not admissible in the accessible graph $\cA$. It follows that the only length-two blocks of the accessible graph $\cA$ which start with
$
C_4\quad\substack{1+\\ \longrightarrow }\quad F_0
$
are
\begin{eqnarray*}
&&
H_1:=
C_4\quad\substack{1+\\ \longrightarrow }\quad
F_0\quad\substack{1-\\ \longrightarrow }\quad
C_3
\\
&&
H_2:=
C_4\quad\substack{1+\\ \longrightarrow }\quad
F_0\quad\substack{5-\\ \longrightarrow }\quad
C_6.
\end{eqnarray*}
For the same reason, the only length-two blocks of the accessible graph $\cA$ which end with
$
C_4\quad\substack{1+\\ \longrightarrow }\quad F_0
$
are
\begin{eqnarray*}
&&
H_3:=
C_3\quad\substack{1-\\ \longrightarrow }\quad
C_4\quad\substack{1+\\ \longrightarrow }\quad
F_0
\\
&&
H_4:=
E_3\quad\substack{5-\\ \longrightarrow }\quad
C_4\quad\substack{1+\\ \longrightarrow }\quad
F_0.
\end{eqnarray*}
According to Lemma \ref{lemObstructionForMultiplicity1}, we have
$
L^{inf}\big(H_1\big)\geq\eta_3>\eta_1
$
and
$
L^{inf}\big(H_3\big)\geq\eta_3>\eta_1
$.
Therefore, if the path $(X_j,\alpha)$ contain infinitely many times $F_0$, then, eventually, at any occurrence of $F_0$ it has to contain the block
$$
E_3\quad\substack{5-\\ \longrightarrow }\quad
C_4\quad\substack{1+\\ \longrightarrow }\quad
F_0\quad\substack{5-\\ \longrightarrow }\quad
C_6.
$$
Moreover
$
C_6\quad\substack{1+\\ \longrightarrow }\quad E_3
$
is both the only arrow of $\cI$ starting at $C_6$ and the only arrow in $\cI$ ending in $E_3$. Therefore, at any occurrence of $F_0$ in the path $(X_j,\alpha)$ we must eventually have the block
$$
C_6\quad\substack{1+\\ \longrightarrow }\quad
E_3\quad\substack{5-\\ \longrightarrow }\quad
C_4\quad\substack{1+\\ \longrightarrow }\quad
F_0\quad\substack{5-\\ \longrightarrow }\quad
C_6\quad\substack{1+\\ \longrightarrow }\quad
E_3.
$$
Any path containing infinitely many times the block above must contain infinitely many time one of the blocks below
\begin{eqnarray*}
&&
H'_1:=
F_0\quad\substack{1+\\ \longrightarrow }\quad
C_5\quad\substack{4-\\ \longrightarrow }\quad
C_6\quad\substack{1+\\ \longrightarrow }\quad
E_3\quad\substack{5-\\ \longrightarrow }\quad
C_4\quad\substack{1+\\ \longrightarrow }\quad
F_0\quad\substack{5-\\ \longrightarrow }\quad
C_6\quad\substack{1+\\ \longrightarrow }\quad
E_3
\\
&&
H'_2:=
C_5\quad\substack{2+\\ \longrightarrow }\quad
C_5\quad\substack{4-\\ \longrightarrow }\quad
C_6\quad\substack{1+\\ \longrightarrow }\quad
E_3\quad\substack{5-\\ \longrightarrow }\quad
C_4\quad\substack{1+\\ \longrightarrow }\quad
F_0\quad\substack{5-\\ \longrightarrow }\quad
C_6\quad\substack{1+\\ \longrightarrow }\quad
E_3
\\
&&
H'_3:=
F_0\quad\substack{5-\\ \longrightarrow }\quad
C_6\quad\substack{1+\\ \longrightarrow }\quad
E_3\quad\substack{5-\\ \longrightarrow }\quad
C_4\quad\substack{1+\\ \longrightarrow }\quad
F_0\quad\substack{5-\\ \longrightarrow }\quad
C_6\quad\substack{1+\\ \longrightarrow }\quad
E_3.
\end{eqnarray*}
The Lemma follows observing that
$$
L^{inf}_5\big(H'_3\big)
\geq
L^{inf}_5\big(H'_1\big)
\geq
L^{inf}_5\big(H'_2\big)
\geq
7\cdot
\frac{[1,4,2,\overline{1,5}]+5+[1,5,1,\overline{1,5}]}{4}.
$$
\end{proof}

\begin{corollary}
\label{CoroDisconnettingIntermediateGraph}
Consider data $(X_j,\alpha)$ such that $L(X_j,\alpha)<\eta_1$. The arrows that can appear infinitely many times in the path generated by $(X_j,\alpha)$ are either only the six arrows
$$
\cF_1:=
C_6\quad\substack{1+\\ \longrightarrow }\quad E_3
\quad
\textrm{ , }
\quad
\cF_2:=
C_5\quad\substack{2+\\ \longrightarrow }\quad C_5
\quad
\textrm{ , }
\quad
\cF_3:=
E_3\quad\substack{3-\\ \longrightarrow }\quad C_6
$$
$$
\cF_{C4}:=
C_5\quad\substack{4-\\ \longrightarrow }\quad C_6
\quad
\textrm{ , }
\quad
\cF_{E4}:=
E_3\quad\substack{4-\\ \longrightarrow }\quad C_5
\quad
\textrm{ , }
\quad
\cG_{C5}:=
C_5\quad\substack{5-\\ \longrightarrow }\quad C_5
$$
or only their image under the involution $\phi$. 
\end{corollary}

\begin{proof}
According to Lemma \ref{LemCutFzero}, any arrow of the intermediate graph $\cI$ which ends in $F_0$ cannot be contained infinitely many times in $(X_j,\alpha)$, thus $F_0$ does not belong to $\support(X_j,\alpha)$. But all arrows of the intermediate graph $\cI$ starting at $C_4$ end in $F_0$, thus $C_4$ does not belong to $\support(X_j,\alpha)$. The statement follows recalling Lemma \ref{lemCutEntriesAbove5}.
\end{proof}

According to Corollary \ref{CoroDisconnettingIntermediateGraph}, in order to finish the  proof of Proposition \ref{propReductionToSmallgraph}, it is enough to exclude the elementary operation $C_5\quad\substack{5-\\ \longrightarrow }\quad C_5$, or in other words the arrow $\cG_{C5}$. This follows from the next Lemma.

\begin{lemma}
\label{lemCutEntriesAbove4inSmallGraph}
We have
$$
L_5^{inf}
\big(
C_5\quad\substack{5-\\ \longrightarrow }\quad C_5
\big)
\geq
\eta_1
$$
\end{lemma}

\begin{proof}
Observe first that
$$
\eta_1<
7\cdot\big(
[1,5,2,4,\overline{1,3}]+[1,4,\overline{1,3}]
\big)
=11,706478\pm10^{-6}.
$$
If
$
L_5^{inf}
\big(
C_5\quad\substack{5-\\ \longrightarrow }\quad C_5
\big)
<\eta_1
$
then there exists $X_j\in\cO$ and $\alpha=[a_1,a_2,\dots]$ irrational with $1\leq a_n\leq5$ for any $n$ which generate a path containing
$
C_5\quad\substack{5-\\ \longrightarrow }\quad C_5
$
infinitely many times and such that $L^{inf}(X_j,\alpha)<\eta_1$. According to Corollary \ref{CoroDisconnettingIntermediateGraph}, the only elementary operations starting in $C_5$ that can appear infinitely many time in $(X_j,\alpha)$ are 
$
C_5\quad\substack{5-\\ \longrightarrow }\quad C_5
$
and
$
C_5\quad\substack{4-\\ \longrightarrow }\quad C_6
$, 
and similarly the only elementary operations ending in $C_5$ that can appear infinitely many times in $(X_j,\alpha)$ are
$
C_5\quad\substack{2+\\ \longrightarrow }\quad C_5
$
and
$
E_3\quad\substack{4-\\ \longrightarrow }\quad C_5
$.
Therefore, if the arrow
$
C_5\quad\substack{5-\\ \longrightarrow }\quad C_5
$
appears infinitely many time in the path generated by $(X_j,\alpha)$, then at any occurrence it also occurs the block
$$
C_5\quad\substack{2+\\ \longrightarrow }\quad
C_5\quad\substack{5-\\ \longrightarrow }\quad
C_5\quad\substack{2+\\ \longrightarrow }\quad
C_5.
$$
We apply Formula \eqref{eqFormulaLagrangeConstant}. Fix $\delta>0$ and consider $N$ such that
$
g(a_1,\dots,a_{N-1})\cdot X_j=C_5
$
and
$a_N=2$, $a_{N+1}=5$, $a_{N+2}=2$ and moreover
\begin{eqnarray*}
&&
L(X_j,\alpha)=
7\cdot
\limsup_{n\to\infty}
\max_{1\leq i\leq a_n}
\frac{D(n,i,\alpha)}
{m^2\big(R\cdot g(a_1,\dots,a_{n-1},i)\cdot X_j\big)}
>
\\
&&
7\cdot
\frac{D(N,1,\alpha)}
{m^2\big(R\cdot g(a_1,\dots,a_{N-1},1)\cdot X_j\big)}+\delta
\end{eqnarray*}
Since
$
g(a_1,\dots,a_{N-1})\cdot X_j=C_5
$,
we have
$
g(a_1,\dots,a_{N-1},1)\cdot X_j=C_4
$
and recalling that $m^2\big(R\cdot C_4\big)=1$ we have
$$
\frac{D(N,1,\alpha)}
{m^2\big(R\cdot g(a_1,\dots,a_{N-1},1)\cdot X_j\big)}
=
[1,a_{N-1},a_{N-2},\dots]+[1,5,2,a_{N+3},a_{N+4},\dots].
$$
According to the shape of the small graph $\cS$ the possible values for $a_{N+3}$ are $a_{N+3}=4$ and $a_{N+3}=5$. Similarly, the possible values for $a_{N-1}$ are $a_{N-1}=4$ and $a_{N-1}=5$. According to Lemma \ref{LemmaMaxMinCantorset} we have
$$
D(N,1,\alpha)
\geq
[1,4,a_{N-2},\dots]+[1,5,2,4,a_{N+4},\dots].
$$
Moreover, if $a_{N-1}=4$ and $a_{N+3}=4$, then we must have
$
g(a_1,\dots,a_{N-2})\cdot X_j=E_3
$
and
$
g(a_1,\dots,a_{N+3})\cdot X_j=C_6
$,
according to the shape of $\cS$. The only arrow of $\cS$ ending in $E_3$ is
$
C_6\quad\substack{1+\\ \longrightarrow }\quad E_3
$, therefore
$$
D(N,1,\alpha)
\geq
[1,4,1,a_{N-3},\dots]+[1,5,2,4,a_{N+4},\dots]
$$
with
$
g(a_1,\dots,a_{N-3})\cdot X_j=
g(a_1,\dots,a_{N+3})\cdot X_j=C_6
$,
in other words we have the block
$$
C_6\quad\substack{1+\\ \longrightarrow }\quad
E_3\quad\substack{4-\\ \longrightarrow }\quad
C_5\quad\substack{2+\\ \longrightarrow }\quad
C_5\quad\substack{5-\\ \longrightarrow }\quad
C_5\quad\substack{2+\\ \longrightarrow }\quad
C_5\quad\substack{4-\\ \longrightarrow }\quad
C_6.
$$
According to the shape of the small graph $\cS$ and to Lemma \ref{LemmaMaxMinCantorset}, in order to minimize $D(N,1,\alpha)$ we must complete the block above repeating both in the past and in the future the block
$
C_6\quad\substack{1+\\ \longrightarrow }\quad
E_3\quad\substack{3-\\ \longrightarrow }\quad
C_6
$.
The Lemma follows recalling that $\delta>0$ is arbitrarily small and observing that we get
$$
L(X_j,\alpha)
>
7\cdot D(N,1,\alpha)+\delta
\geq
7\cdot\big(
[1,4,\overline{1,3}]+[1,5,2,4,\overline{1,3}]\big)+\delta
>\eta_1,
$$
which is absurd.
\end{proof}

\subsection{Proof of Theorem \ref{TheoremReductionToSubshift}}
\label{Sec:EndProofTheoremREductionToSubshift}

The proof splits in two parts.
\subsubsection{Coding by the subshift}

Let $(X_j,\alpha)$ be data with $L(X_j,\alpha)<\eta_1$. According to Proposition \ref{propReductionToSmallgraph}, modulo replacing the path  by its image under the involution $\psi$,   the arrows that can appear infinitely many times in the path generated by $(X_j,\alpha)$,  are the five arrows
$\cF_1$, $\cF_2$, $\cF_3$, $\cF_{C4}$, $\cF_{E4}$ of the small graph $\cS$.  In particular, any path as above necessarily passes infinitely many times from the vertex $C_6$, thus it can be eventually decomposed as $H_0\ast H_1\ast H_2\dots$, where any $H_i$ is one of the two elementary loops $H_a$ and $H_b$ at $C_6$ which are defined below and represented in Figure \ref{FigureLoopsHaHb}.
\begin{eqnarray*}
&&
H_a:=
C_6\quad\substack{1+\\ \longrightarrow }\quad
E_3\quad\substack{4-\\ \longrightarrow }\quad
C_5\quad\substack{2+\\ \longrightarrow }\quad
C_5\quad\substack{4-\\ \longrightarrow }\quad
C_6
\\
&&
H_b:=
C_6\quad\substack{1+\\ \longrightarrow }\quad
E_3\quad\substack{3-\\ \longrightarrow }\quad
C_6.
\end{eqnarray*}

\begin{figure}[!ht]
\label{FigureLoopsHaHb}
\includegraphics[width=.2\textwidth]{small-graph.pdf}
\hspace{2.3mm}
\includegraphics[width=.2\textwidth]{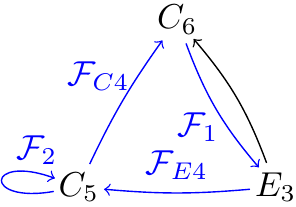}
\hspace{2.3mm}
\includegraphics[width=.2\textwidth]{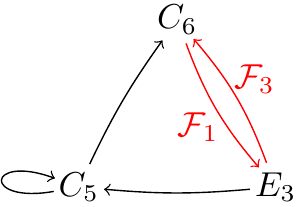}
\caption{The \emph{small graph} $\cS$ and the two elementary loops $H_a$ and $H_b$.\label{smallgraph}}
\end{figure}

One can also see that those paths which pass from $C_5$ just finitely many times eventually coincide with the periodic loop $(C_6,[\overline{1,3}])$, and we alredy know that $L(C_6,[\overline{1,3}])=\phi_1$. Therefore we consider without loss of generality the set $\Gamma$ of paths $\gamma$ starting at $C_6$, composed by the five arrows above, containing infinitely many times both the loops $H_a$ and $H_b$, and whose first loop is $H_0=H_a$. Consider the two finite words $a:=1,4,2,4$ and $b:=1,3$ and let $\sigma:\Xi_0\to\Xi_0$ be the map defined in the introduction, where $\Xi_0$ is the set of $\{a,b\}^\ZZ$ of those sequences $\xi=(\xi_i)_{i\in\ZZ}$ such that $\xi_0=a$ and also $\xi_i=a$ for arbitrarily big integers $i>0$. Define a map $\Pi:\Gamma\to\Xi_0$, where for any $\gamma\in\Gamma$ the sequence
$
(\xi_i)_{i\in\ZZ}=\Pi(\gamma)
$
is defined setting $\xi_i:=a$ for $i\leq0$ and for all positive integer $i$ setting
\begin{eqnarray*}
&&
\xi_i:=a
\textrm{ if }
H_i=H_a
\textrm{ and }
\\
&&
\xi_i:=b
\textrm{ if }
H_i=H_b,
\end{eqnarray*}
where
$
\gamma=H_1\ast H_2\ast H_3\ast\dots
$
is the decomposition of $\gamma$ in loops $H_i\in\{H_a,H_b\}$. Recall that for any sequence $\xi\in\Xi_0$ we write
$
[\xi]_{+}:=[1,4,\xi_1,\xi_2,\dots]
$
and
$
[\xi]_{-}:=[1,4,\xi_{-1},\xi_{-2},\dots]
$.
Consider the words $<a>:=4,2,4,1$ and $<b>:=3,1$, then define an operation $<\cdot>$ on the set of finite words $u=\xi_1,\dots,\xi_k$ in the letters $a,b$ setting
$$
<\xi_1,\dots,\xi_k>:=<\xi_1>,\dots,<\xi_k>.
$$
Observe that $1,<a>=1,4,2,4,1=a,1$ and $1,<b>=1,3,1=b,1$, thus for any finite word $u$ we have the identity
$$
1,<u>=u,1.
$$
Finally, given two sequences $(x_n)_{n\in\NN}$ and $(y_n)_{n\in\NN}$ we write $x_n\sim y_n$ if $x_n-y_n\to0$ for $n\to+\infty$.

\begin{lemma}
\label{LemCodingBySubshift}
Let $(X_j,\alpha)$ be data generating a path $\gamma\in\Gamma$ and set $\xi=\Pi(\gamma)$. Then at any occurrence of $\cF_2$, that is for any $n$ with
$
g(a_1,\dots,a_{n-1})\cdot X_j=C_5
$
and $a_n=2$ we have
$$
\frac{D(n,1,\alpha)}
{m^2\big(R\cdot g(a_1,\dots,a_{n-1},1)\cdot X_j\big)}
\sim
\big[\sigma^{k(n)}(\xi)\big]_{+}
+
\big[\sigma^{k(n)}(\xi)\big]_{-},
$$
where the positive integer $k(n)$ corresponds to the number of occurrences of the loop $H_a$ in the segment of the path $\gamma$ corresponding to $g(a_1,\dots,a_n)$. 
\end{lemma}

\begin{proof}
Observe that any occurrence in the path $\gamma$ of the arrow
$
\cF_2=C_5\quad\substack{2+\\ \longrightarrow }\quad C_5
$
corresponds to an occurrence of the block
$$
C_6\quad\substack{1+\\ \longrightarrow }\quad
E_3\quad\substack{4-\\ \longrightarrow }\quad
C_5\quad\substack{2+\\ \longrightarrow }\quad
C_5\quad\substack{4-\\ \longrightarrow }\quad
C_6\quad\substack{1+\\ \longrightarrow }\quad
E_3.
$$
Therefore, for any integer $n$ as in the statement we have $a_{n-2}=1$, $a_{n-1}=4$, $a_{n}=2$, $a_{n+1}=4$ and $a_{n+2}=1$. Moreover we have
$
g(a_1,\dots,a_{n-1},1)\cdot X_j=C_4
$,
and since $m^2\big(R\cdot C_4\big)=1$ then we get
$$
\frac{D(n,1,\alpha)}
{m^2\big(R\cdot g(a_1,\dots,a_{n-1},1)\cdot X_j\big)}
=
[1,4,1,a_{n-3},\dots,a_1]
+
[1,4,1,a_{n+3},\dots].
$$
Let $k=k(n)$ be the integer in the statement, which corresponds to the number of occurrences of the loop $H_a$ in the segment of the path $\gamma$ corresponding to $g(a_1,\dots,a_n)$. If
$
\gamma=H_0\ast H_1\ast H_2\dots
$
is the decomposition of $\gamma$ into loops $H_i\in\{H_a,H_b\}$, then the map $\Pi$ gives
$$
[1,4,1,a_{n+3},\dots]
=
[1,4,\xi_{k+1},\xi_{k+2}\dots]
=
\big[\sigma^{k}(\xi)\big]_{+}.
$$
With the same argument, and recalling that any finite word $u$ in the letters $a,b$ satisfies the identity $1,<u>=u,1$, we have
$$
[1,4,1,a_{n-3},\dots,a_1]
=
[1,4,1,<\xi_{k-1}>,\dots,<\xi_0>]
=
[1,4,\xi_{k-1},\dots,\xi_0,1]
\sim
\big[\sigma^{k}(\xi)\big]_{-}.
$$
\end{proof}

\subsubsection{End of the proof of Theorem \ref{TheoremReductionToSubshift}}

Let $(X_j,\alpha)$ be data with $L(X_j,\alpha)<\eta_1$ and assume without loss of generality that they generate a path $\gamma\in\Gamma$.

\medskip

We have $m^2\big(R\cdot C_6\big)=1$ and the only arrow starting at $C_6$ is $\cF_1$. According to Formula \eqref{eqFormulaLagrangeConstant} and Lemma \ref{LemmaMaxMinCantorset}, for the Gauss approximations corresponding to occurrences of $\cF_1$ we have the estimation
$$
\frac{D(n,1,\alpha)}
{m^2\big(R\cdot g(a_1,\dots,a_{n})\cdot X_j\big)}
<
[\overline{3,1}]+1+[\overline{3,1}]=
\phi_1=1,527524\pm10^{-6}.
$$

We have
$
m^2\big(R\cdot C_5\big)=
m^2\big(R\cdot E_3\big)=4
$
for all the endpoints of all the remaining four arrows $\cF_2$, $\cF_3$, $\cF_{C4}$ and $\cF_{E4}$. Moreover $1\leq a_n\leq4$ for any $n$ big enough, therefore for Gauss approximations Formula \eqref{eqFormulaLagrangeConstant} gives the estimation
$$
\frac{D(n,a_n,\alpha)}
{m^2\big(R\cdot g(a_1,\dots,a_{n})\cdot X_j\big)}
<
\frac{1+4+1}{4}=1,5.
$$

Since
$
m^2\big(C_2\big)=m^2\big(C_1\big)=m^2\big(C_0\big)=m^2\big(C_6\big)=1
$,
then Formula \eqref{eqFormulaLagrangeConstant} and Lemma \ref{LemmaMaxMinCantorset} give for the Farey approximations corresponding to occurrences of the arrows $\cF_{C4}$ and $\cF_{E4}$
\begin{eqnarray*}
&&
\frac{D(n,1,\alpha)}
{m^2\big(R\cdot g(a_1,\dots,a_{n})\cdot X_j\big)}
<
[1,\overline{1,4}]+[3,\overline{1,4}]<1
\\
&&
\frac{D(n,2,\alpha)}
{m^2\big(R\cdot g(a_1,\dots,a_{n})\cdot X_j\big)}
<
2[2,\overline{1,4}]<1
\\
&&
\frac{D(n,3,\alpha)}
{m^2\big(R\cdot g(a_1,\dots,a_{n})\cdot X_j\big)}
<
[3,\overline{1,4}]+[1,\overline{1,4}]<1.
\end{eqnarray*}

Again Formula \eqref{eqFormulaLagrangeConstant} and Lemma \ref{LemmaMaxMinCantorset} give for the Farey approximations corresponding to occurrences of the arrow $\cF_{3}$
\begin{eqnarray*}
&&
\frac{D(n,1,\alpha)}
{m^2\big(R\cdot g(a_1,\dots,a_{n})\cdot X_j\big)}
<
[1,1,\overline{3,1}]+[2,1,\overline{3,1}]<1
\\
&&
\frac{D(n,2,\alpha)}
{m^2\big(R\cdot g(a_1,\dots,a_{n})\cdot X_j\big)}
<
[2,1,\overline{3,1}]+[1,1,\overline{3,1}]<1.
\end{eqnarray*}

So far we proved that for all Gauss approximations and all Farey approximations but those corresponding to the occurrences of the arrow $\cF_2$ we have
$$
\frac{D(n,1,\alpha)}
{m^2\big(R\cdot g(a_1,\dots,a_{n})\cdot X_j\big)}
<
\phi_1\sim1,527524\pm10^{-6}.
$$
According to Lemma \ref{LemCodingBySubshift} and to Lemma \ref{LemmaMaxMinCantorset}, for any Farey approximation corresponding to an occurrence of $\cF_2$, that is for those $n$ such that $g(a_1,\dots,a_{n-1})=C_5$ and $a_n=2$, we have
\begin{eqnarray*}
&&
\frac{D(n,1,\alpha)}
{m^2\big(R\cdot g(a_1,\dots,a_{n-1},1)\cdot X_j\big)}
=
[1,4,1,a_{n-3},\dots,a_1]
+
[1,4,1,a_{n+3},\dots]
\\
&&
\geq
[1,4,\overline{1,3}]
+
[1,4,\overline{1,3}]
=
1,654653\pm10^{-6}.
\end{eqnarray*}
Theorem \ref{TheoremReductionToSubshift} is proved, since the latter is the bigger of all the values computed above.

\section{Proof Theorem \ref{TheoremMainTheorem}}
\label{Sec:ProofOfMainTheorem}

In this paragraph we prove Theorem \ref{TheoremMainTheorem} using the function $L^\sigma:\Xi_0\to\RR_+$ in Theorem \ref{TheoremReductionToSubshift}. Let $\KK:=L^\sigma(\Xi_0)$ be the set of values of the function. We believe that Theorem \ref{TheoremMainTheorem} can be improved proving that $\KK$ is a cantor set. Therefore our proof is presented as an iterative construction of a Cantor set as $\KK:=\bigcap_{n=1}^\infty\KK(n)$ where each $\KK(n)$ is a countable union of closed intervals satisfying $\KK(n+1)\subset\KK(n)$ for any $n\geq1$. The two families $\KK(1)$ and $\KK(2)$ are described in \S \ref{SecFirstGenerationCantor} and \S \ref{SecSecondGenerationCantor} respectively. Finally in \S~\ref{EndOfTheProofOfMainTheorem} we briefly rephrase results as in the statement of Theorem \ref{TheoremMainTheorem}.

\subsection{Lexicographic order on finite words in the letters $a,b$}
\label{Sec:LexicographicOrder}

Let $u$ be a finite word in the letters $a,b$. Let $u^{(k)}$ be the finite word obtained as concatenation $u\ast\dots\ast u$ of $k$ copies of the word $u$. Let $u^{(\infty,+)}$ and $u^{(\infty,-)}$ be respectively the positive infinite word and the negative infinite word obtained concatenating periodically the word $u$. Let $u^\infty$ be the periodic infinite word whose period is $u$, so that for instance we have
$
u^\infty=u^{(\infty,-)}\ast u^{(\infty,+)}
$.
Let $v$ be an other finite word in the letter $a,b$. Let
$
\overline{v^\infty uv^\infty}
$
be the infinite word obtained by concatenation
$$
\overline{v^\infty uv^\infty}
:=
u^{(\infty,-)}
\ast
v\ast u\ast v
\ast
v^{(2)}\ast u\ast v^{(2)}
\ast\dots\ast
v^{(k)}\ast u\ast v^{(k)}
\ast\dots,
$$
that is, the infinite word whose future is the concatenation of blocks $v^{(k)}\ast u\ast v^{(k)}$ with increasing values of $k$ and whose past is the half-periodic word $u^{(\infty,-)}$. Actually, since we will consider iterates $\sigma^n(\xi)$ with $n\to+\infty$ of sequences $\xi\in\Xi_0$, the past of $\xi$ does not affect $L^\sigma(\xi)$, and thus in particular one could define $\overline{v^\infty uv^\infty}$ replacing $u^{(\infty,-)}$ by any other half-infinite sequence in the letters $a,b$.

\medskip

We first give a numerical result, which depends on the explicit values of the finite words $a,b$.

\begin{lemma}
\label{LemComparisonAandB}
We have the inequalities
\begin{eqnarray*}
&&
2[a,b^{(\infty,+)}]
>
[b,a^{(\infty,+)}]
+
[a^{(\infty,+)}]
\\
&&
2[b,a^{(\infty,+)}]
<
[a,b^{(\infty,+)}]
+
[b^{(\infty,+)}].
\end{eqnarray*}
\end{lemma}

\begin{proof}
The Lemma just follows comparing the values
\begin{eqnarray*}
&&
[b^{(\infty,+)}]=0,79128784 \pm 10^{-8}
\\
&&
[b,a^{(\infty,+)}]=0,79238557 \pm 10^{-8}
\\
&&
[a,b^{(\infty,+)}]=0,81660638 \pm 10^{-8}
\\
&&
[a^{(\infty,+)}]=0,81661395 \pm 10^{-8}.
\end{eqnarray*}
\end{proof}

\begin{lemma}
[Lexicographic Order]
\label{LemLexicographicOrder}
Let $u$ be a finite word in the letters $a,b$. For any positive infinite words $\omega$ and $\omega'$ in the letters $a,b$ we have the strict inequality
$$
[1,4,u,b,\omega]<[1,4,u,a,\omega'].
$$
\end{lemma}

\begin{proof}
Let $g\in\psltwoz$ be the homography such that
$
g\big([\omega'']\big)=[1,4,u,\omega'']
$
for any positive infinite word $\omega''$. The Lemma corresponds to prove that
$$
g\big([b,\omega]\big)<g\big([a,\omega']\big)
$$
for any positive infinite words $\omega$ and $\omega'$. Since $g$ is increasing monotone, the last inequality is equivalent to
$[b,\omega]<[a,\omega']$, which follows from Lemma \ref{LemComparisonAandB} observing that
$$
[b,\omega]
\leq
[b,a^{(\infty,+)}]
<
[a,b^{(\infty,+)}]
\leq
[a,\omega'].
$$
\end{proof}

In the following we will use frequently the Lemma below, whose simple proof is left to the reader.

\begin{lemma}
\label{lemdistortionhomography}
Let $g:\RR_+\to\RR_+$ be a concave and increasing monotone function. If $x_1,x_2,x_3,x_4$ are points in $\RR_+$ such that
\begin{eqnarray*}
&&
x_1=\min_{i=1,2,3,4} x_i
\\
&&
x_4=\max_{i=1,2,3,4} x_i
\\
&&
x_2+x_3>x_1+x_4.
\end{eqnarray*}
Then we have
$$
g(x_2)+g(x_3)>g(x_1)+g(x_4).
$$
\end{lemma}

Let $u$ be any finite word in the letters $a,b$ and observe that the assumption of Lemma \ref{lemdistortionhomography} are satisfied by the homography $g$ such that
$
g\big([\omega]\big)=[1,4,u,\omega]
$
for any positive infinite words $\omega$, indeed $g$ is increasing as element of $\psltwoz$, moreover its pole $g^{-1}(\infty)$ is negative, since all coefficients of $g$ are positive, thus $g$ is positive and concave on $\RR_+$.

\begin{lemma}
\label{LemCenteringInTheMiddle}
Consider a pair of integers $i,j\in\NN$ such that either $i=j$ or $|i-j|=1$. Then for any pair of integers $i',j'\in\NN$ with $|i'+j'|=|i+j|$ and $|i'-j'|>|i-j|$ and for any positive infinite words $\omega$ and $\omega'$ in the letters $a,b$ we have the strict inequality
$$
[1,4,a^{(i)},b,\omega]+
[1,4,a^{(j)},b,\omega']
>
[1,4,a^{(i')},b,\omega]+
[1,4,a^{(j')},b,\omega']
$$
\end{lemma}

\begin{proof}
Assume without loss of generality that $j'<j\leq i<i'$ and consider the positive integer $l:=i'-i\geq1$. Observe that we have also $j-j'=l$ since $i+j=i'+j'$. Consider the homography $g$ such that
$
g\big([\omega'']\big)=[1,4,a^{(j')},\omega'']
$
for any positive infinite word $\omega''$. Observe that $g$ is increasing monotone, since it belongs to $\psltwoz$, and that $g^{-1}(\infty)<0$, since $g$ has positive coefficients. The Lemma follows from
\begin{eqnarray*}
&&
g\big([a^{(l)},b,\omega]\big)+
g\big([a^{(l)},b,\omega']\big)
>
g\big([a^{(2l)},b,\omega]\big)+
g\big([b,\omega']\big)
\textrm{ if }
i=j
\\
&&
g\big([a^{(l+1)},b,w]\big)+
g\big([a^{(l)},b,w']\big)
>
g\big([a^{(2l+1)},b,w]\big)+
g\big([b,w']\big)
\textrm{ if }
i=j+1.
\end{eqnarray*}
According to Lemma \ref{LemLexicographicOrder} and recalling that $g$ is increasing monotone, it is enough to prove that for any $l\geq1$ we have
\begin{eqnarray*}
&&
g\big([a^{(l)},b^{(\infty,+)}]\big)+
g\big([a^{(l)},b^{(\infty,+)}]\big)
>
g\big([a^{(2l)},b,a^{(\infty,+)}]\big)+
g\big([b,a^{(\infty,+)}]\big)
\\
&&
g\big([a^{(l+1)},b^{(\infty,+)}]\big)+
g\big([a^{(l)},b^{(\infty,+)}]\big)
>
g\big([a^{(2l+1)},b,a^{(\infty,+)}]\big)+
g\big([b,a^{(\infty,+)}]\big)
\end{eqnarray*}
and since $g$ is concave and strictly increasing, according to Lemma \ref{lemdistortionhomography} it is enough to prove that for any $l\geq1$ we have \begin{eqnarray*}
&&
[a^{(l)},b^{(\infty,+)}]
+
[a^{(l)},b^{(\infty,+)}]
>
[a^{(2l)},b,a^{(\infty,+)}]
+
[b,a^{(\infty,+)}]
\\
&&
[a^{(l+1)},b^{(\infty,+)}]
+
[a^{(l)},b^{(\infty,+)}]
>
[a^{(2l+1)},b,a^{(\infty,+)}]
+
[b,a^{(\infty,+)}].
\end{eqnarray*}
Since $l\geq1$, according to Lemma \ref{LemLexicographicOrder} it is enough to prove
\begin{eqnarray*}
&&
2[a^{(1)},b^{(\infty,+)}]
>
[a^{(\infty,+)}]
+
[b,a^{(\infty,+)}]
\\
&&
[a^{(2)},b^{(\infty,+)}]
+
[a^{(1)},b^{(\infty,+)}]
>
[a^{(\infty,+)}]
+
[b,a^{(\infty,+)}].
\end{eqnarray*}
The first inequality corresponds to the statement of Lemma \ref{LemComparisonAandB}. The second inequality follows from the first observing that
$
[a^{(2)},b^{(\infty,+)}]>[a^{(1)},b^{(\infty,+)}]
$,
according to Lemma \ref{LemLexicographicOrder}.
\end{proof}

\subsection{The first generation of $\KK$}\label{SecFirstGenerationCantor}

Consider the function $\kappa:\Xi_0\to\NN^\ast$, where for any $\xi\in\Xi_0$ the integer $\kappa(\xi)$ is the maximum $k\geq1$ such that for any $N\in\NN$ there exists $n>N$ with
$\xi_j=a$ for all $j$ with $n\leq j\leq n+k-1$.

\begin{lemma}
\label{lemLagrangeConstantOnFirstGeneration}
Consider $\xi\in\Xi_0$ such that $\kappa(\xi)=k$. Then there exist positive infinite words $\omega=\omega(\xi)$ and $\omega'=\omega'(\xi)$ in the letters $a,b$ such that
\begin{eqnarray*}
&&
L^\sigma(\xi)=
7\cdot\big(
[1,4,a^{(i)},b,\omega']
+
[1,4,a^{(i)},b,\omega]
\big)
\textrm{ if }
k=2i+1
\\
&&
L^\sigma(\xi)=
7\cdot\big(
[1,4,a^{(i+1)},b,\omega']
+
[1,4,a^{(i)},b,\omega]
\big)
\textrm{ if }
k=2i+2.
\end{eqnarray*}
\end{lemma}

\begin{proof}


There exist two positive infinite words $\omega$ and $\omega'$ such that
$$
L^\sigma(\xi)=7\cdot\big([1,4,\omega']+[1,4,\omega]\big).
$$
The Lemma follows from Lemma \ref{LemLexicographicOrder} and Lemma \ref{LemCenteringInTheMiddle}.
\end{proof}

We observe that the two terms containing respectively in $\omega$ and $\omega'$ in Lemma \ref{lemLagrangeConstantOnFirstGeneration} does not necessarily correspond to the \emph{past} or to the \emph{future}. The same ambiguity appears further in Lemma \ref{lemLagrangeConstantOnSecondGeneration}.

\begin{lemma}
\label{LemFirstGenerationCantor}
The first generation $\KK(1)$ is the set of closed intervals $I_k$ with $k\in\NN^\ast$ defined by
$$
I_k:=
\big[
L^\sigma(\overline{b^\infty a^{(k)}b^\infty}),
L^\sigma((ba^{(k)})^\infty)
\big].
$$
Moreover, for any $k\in\NN$ and any $\xi\in\Xi_0$ we have $L^\sigma(\xi)\in I_k$ if and only if $\kappa(\xi)=k$. Finally, the last two conditions are equivalent to the existence of positive infinite words $\omega$ and $\omega'$, depending on $\xi$ and $k$, such that
\begin{eqnarray*}
&&
L^\sigma(\xi)=
7\cdot\big(
[1,4,a^{(i)},b,\omega']
+
[1,4,a^{(i)},b,\omega]
\big)
\textrm{ if }
k=2i+1
\\
&&
L^\sigma(\xi)=
7\cdot\big(
[1,4,a^{(i)},b,\omega']
+
[1,4,a^{(i+1)},b,\omega]
\big)
\textrm{ if }
k=2i+2.
\end{eqnarray*}
\end{lemma}

As it is described in \S \ref{EndOfTheProofOfMainTheorem} below, the first generation of gaps $(G_k)_{k\geq1}$ is given by the connected components of the complement of $\KK(1)$.

\begin{proof}
We prove that for any integer $k\geq 0$ we have a gap in $\KK$ given by the open interval
$$
\big(
L^\sigma((ba^{(k)})^\infty),
L^\sigma(\overline{b^\infty a^{(k+1)}b^\infty})
\big).
$$
We assume that $k$ is odd, the proof for even $k$ being the same. Set $k=2i+1$ with $i\in\NN$. Since both
$
\kappa\big(
L^\sigma(b(a^{(k)})^\infty)
\big)=k
$
and
$
\kappa\big(
\overline{b^\infty a^{(k+1)}b^\infty}
\big)=k
$,
Lemma \ref{lemLagrangeConstantOnFirstGeneration} implies
\begin{eqnarray*}
&&
L^\sigma\big((ba^{(k)})^\infty\big)
=
7\cdot\big(
[1,4,a^{(i)},(ba^{(k)})^{(\infty,+)}]
+
[1,4,a^{(i)},(ba^{(k)})^{(\infty,+)}]
\big)
\\
&&
L^\sigma(\overline{b^\infty a^{(k+1)}b^\infty})
=
7\cdot\big(
[1,4,a^{(i)},b^{(\infty,+)}]
+
[1,4,a^{(i+1)},b^{(\infty,+)}]
\big).
\end{eqnarray*}
Comparing the two quantities above with the expression in Lemma \ref{lemLagrangeConstantOnFirstGeneration} one gets
\begin{eqnarray*}
&&
L^\sigma((ba^{(k)})^\infty)
=
\max\{L^\sigma(\xi)
\textrm{ ; }
\xi\in\Xi_0,\kappa(\xi)=k\}
\\
&&
L^\sigma(\overline{b^\infty a^{(k+1)}b^\infty})
=
\min\{L^\sigma(\xi)
\textrm{ ; }
\xi\in\Xi_0,\kappa(\xi)=k+1\}.
\end{eqnarray*}
The Lemma follows proving that we have the strict inequality
$$
L^\sigma\big((ba^{(k)})^\infty\big)
<
L^\sigma(\overline{b^\infty a^{(k+1)}b^\infty}).
$$
According to the formulae for
$
L^\sigma\big((ba^{(k)})^\infty\big)
$
and
$
L^\sigma(\overline{b^\infty a^{(k+1)}b^\infty})
$
obtained above, and observing that Lemma \ref{LemLexicographicOrder} implies
$$
L^\sigma\big((ba^{(k)})^\infty\big)
=
7\cdot\big(
[1,4,a^{(i)},(ba^{(k)})^{(\infty,+)}]
+
[1,4,a^{(i)},(ba^{(k)})^{(\infty,+)}]
\big)
\leq
14\cdot[1,4,a^{(i)},b,a^{(\infty,+)}]
$$
it is enough to prove that
$$
2\cdot[1,4,a^{(i)},b,a^{(\infty,+)}]
<
[1,4,a^{(i)},b^{(\infty,+)}]
+
[1,4,a^{(i+1)},b^{(\infty,+)}].
$$
In order to prove the last inequality, let $g$ be the homography such that
$
g\big([\omega]\big)=[1,4,a^{(i)},\omega]
$
for any positive infinite word $\omega$. The inequality is equivalent to
$$
g\big([b,a^{(\infty,+)}]\big)
-
g\big([b^{(\infty,+)}]\big)
<
g\big([a,b^{(\infty,+)}]\big)
-
g\big([b,a^{(\infty,+)}]\big)
$$
and the latter follows from Lemma \ref{lemdistortionhomography}.
\end{proof}

\subsection{The second generation of $\KK$}\label{SecSecondGenerationCantor}

Let $\Xi_{0,k}$ be the set of $\xi\in\Xi_0$ with $\kappa(\xi)=k$. Define a function $\nu:\Xi_{0,k}\to\NN$, where for any $\xi\in\Xi_{0,k}$ the integer $\nu(\xi)$ is the minimum positive integer $n$ such that in $\xi$ we have infinitely often one of the two finite words
$$
\textrm{ either }
a,b^{(n)},a^{(k)}
\textrm{ or }
a^{(k)},b^{(n)},a.
$$

\begin{lemma}
\label{LemSecondGeerationCantor(Technical)}
For any $i\in\NN$, any positive integer $n$ and any pair of positive infinite words $\omega$ and $\omega'$ in he letters $a,b$ we have
$$
[1,4,a^{(i+1)},b^{(n)},\omega]
+
[1,4,a^{(i)},b^{(n)},a,\omega']
>
[1,4,a^{(i)},b^{(n)},\omega]
+
[1,4,a^{(i+1)},b^{(n)},a,\omega'].
$$
\end{lemma}

\begin{proof}
Let $h$ and $g$ be the homographies such that we have respectively $
h\big([\omega'']\big)=[a,\omega'']
$
and
$
g\big([\omega'']\big)=[1,4,a^{(i)},\omega'']
$
for any positive infinite word $\omega''$. Since $0<h'(t)<1$ for any $h\in\psltwoz$ and any $t>0$ then we have
$$
[b^{(n)},a,\omega']
-
[b^{(n)},\omega]
>
h\big([b^{(n)},a,\omega']\big)
-
h\big([b^{(n)},\omega]\big),
$$
that is
$$
[a,b^{(n)},\omega]+[b^{(n)},a,\omega']
>
[b^{(n)},\omega]+[a,b^{(n)},a,\omega'].
$$
The statement follows applying Lemma \ref{lemdistortionhomography} to the function $g$.
\end{proof}

The same argument on the proof of Lemma \ref{lemLagrangeConstantOnFirstGeneration} and the estimate in Lemma \ref{LemSecondGeerationCantor(Technical)} give the Lemma below.

\begin{lemma}
\label{lemLagrangeConstantOnSecondGeneration}
Consider $\xi\in\Xi_{0,k}$ such that $\nu(\xi)=n$. Then there exist positive infinite words $\omega=\omega(\xi)$ and $\omega'=\omega'(\xi)$ in the letters $a,b$ such that
\begin{eqnarray*}
&&
L^\sigma(\xi)=
7\cdot\big(
[1,4,a^{(i)},b^{(n)},\omega']
+
[1,4,a^{(i)},b^{(n)},a,\omega]
\big)
\textrm{ if }
k=2i+1
\\
&&
L^\sigma(\xi)=
7\cdot\big(
[1,4,a^{(i+1)},b^{(n)},\omega']
+
[1,4,a^{(i)},b^{(n)},a,\omega]
\big)
\textrm{ if }
k=2i+2.
\end{eqnarray*}
\end{lemma}

\begin{lemma}
\label{LemSecondGenerationCantor}
For any $I_k\in\KK(1)$ the generation $\KK(2|I_k)$ is the family of closed intervals $J_n$ for $n\in\NN^\ast$ defined by
$$
J_n=
\big[
L^\sigma(\overline{b^\infty a^{(k)}b^{(n)}ab^\infty}),
L^\sigma((a^{(k)}b^{(n)})^\infty)
\big].
$$
Moreover, for any $n\in\NN^\ast$ and any $\xi\in\Xi_{0,k}$ we have $L^\sigma(\xi)\in J_n$ if and only if $\nu(\xi)=n$. Finally, the last two conditions are equivalent to the existence of positive infinite words $\omega$ and $\omega'$, depending on $\xi$ such that
\begin{eqnarray*}
&&
L^\sigma(\xi)=
7\cdot\big(
[1,4,a^{(i)},b^{(n)},\omega']
+
[1,4,a^{(i)},b^{(n)},a,\omega]
\big)
\textrm{ if }
k=2i+1
\\
&&
L^\sigma(\xi)=
7\cdot\big(
[1,4,a^{(i+1)},b^{(n)},\omega']
+
[1,4,a^{(i)},b^{(n)},a,\omega]
\big)
\textrm{ if }
k=2i+2.
\end{eqnarray*}
\end{lemma}

As it is described in \S \ref{EndOfTheProofOfMainTheorem} below, for any $I_k\in\KK(1)$ the second generation of gaps $(G_{k,n})_{n\geq1}$ in the interval $I_k$ is given by the connected components of the complement of $\KK(2|I_k)$.

\begin{proof}
We prove that for any $n\geq 0$ we have a gap in $I_k$ corresponding to the open interval
$$
\big(
L^\sigma((a^{(k)}b^{(n)})^\infty),
L^\sigma(\overline{b^\infty a^{(k)}b^{(n+1)}ab^\infty})
\big).
$$
We assume that $k$ is even, the proof for even $k$ being the same (the assumption is complementary to the one in the proof of Lemma \ref{LemFirstGenerationCantor}, where only the case of odd $k$ is considered explicitly). Set $k=2i+2$ with $i\in\NN$. Observe that both
$
L^\sigma((a^{(k)}b^{(n)})^\infty)
$
and
$
L^\sigma(\overline{b^\infty a^{(k)}b^{(n+1)}ab^\infty})
$
belong to $I_k$, therefore Lemma \ref{LemFirstGenerationCantor} implies
\begin{eqnarray*}
&&
L^\sigma\big((a^{(k)}b^{(n+1)})^\infty\big)
=
7\cdot\big(
[1,4,a^{(i+1)},(b^{(n+1)}a^{(k)})^{(\infty,+)}]
+
[1,4,a^{(i)},(b^{(n+1)}a^{(k)})^{(\infty,+)}]
\big)
\\
&&
L^\sigma(\overline{b^\infty a^{(k)}b^{(n)}ab^\infty})
=
7\cdot\big(
[1,4,a^{(i+1)},b^{(\infty,+)}]
+
[1,4,a^{(i)},b^{(n)},a,b^{(\infty,+)}]
\big),
\end{eqnarray*}
where in the second equality we also need the estimate in Lemma \ref{LemSecondGeerationCantor(Technical)}. Lemma \ref{lemLagrangeConstantOnSecondGeneration} implies
\begin{eqnarray*}
&&
L^\sigma\big((a^{(k)}b^{(n+1)})^\infty\big)
=
\max\{L^\sigma(\xi)
\textrm{ ; }
\xi\in\Xi_{0,k},\nu(\xi)=n+1\}
\\
&&
L^\sigma(\overline{b^\infty a^{(k)}b^{(n)}ab^\infty})
=
\min\{L^\sigma(\xi)
\textrm{ ; }
\xi\in\Xi_{0,k},\nu(\xi)=n\}.
\end{eqnarray*}
The Lemma follows proving that we have the strict inequality
$$
L^\sigma\big((a^{(k)}b^{(n+1)})^\infty\big)
<
L^\sigma(\overline{b^\infty a^{(k)}b^{(n)}ab^\infty}).
$$
According to the expressions of these two quantities obtained above and observing that Lemma \ref{LemLexicographicOrder} implies
\begin{eqnarray*}
&&
L^\sigma\big((a^{(k)}b^{(n+1)})^\infty\big)
<
7\cdot\big(
[1,4,a^{(i+1)},b^{(n+1)},a^{(\infty,+)}]
+
[1,4,a^{(i)},b^{(n+1)},a^{(\infty,+)}]
\big)
\\
&&
L^\sigma(\overline{b^\infty a^{(k)}b^{(n)}ab^\infty})
=
7\cdot\big(
[1,4,a^{(i+1)},b^{(\infty,+)}]
+
[1,4,a^{(i)},b^{(n)},a,b^{(\infty,+)}]
\big)
\end{eqnarray*}
it is enough to prove
$$
g\big([a,b^{(n+1)},a^{(\infty,+)}]\big)
+
g\big([b^{(n+1)},a^{(\infty,+)}]\big)
<
g\big([a,b^{(\infty,+)}]\big)
+
g\big([b^{(n)},a,b^{(\infty,+)}]\big),
$$
where $g$ is the homography such that
$
g\big([\omega]\big)=[1,4,a^{(i)},\omega]
$
for any positive infinite word $\omega$.
\end{proof}

\subsection{End of the proof of Theorem \ref{TheoremMainTheorem}}\label{EndOfTheProofOfMainTheorem}

Theorem \ref{TheoremMainTheorem} simply follows rephrasing the statements of Lemma \ref{LemFirstGenerationCantor} and of Lemma \ref{LemSecondGenerationCantor}. The first generation $(G_k)_{k\geq1}$ of gaps in $\cL(S)$ is given by
\begin{eqnarray*}
&&
G_0:=\big(\phi_1,\phi_2\big)
\textrm{ and }
\\
&&
G_k:=
\big(
L^\sigma((ba^{(k)})^\infty),
L^\sigma(\overline{b^\infty a^{(k+1)}b^\infty})
\big)
\textrm{ for }
k\geq 1.
\end{eqnarray*}
For any interval
$
I_k=
\big[
L^\sigma(\overline{b^\infty a^{(k)}b^\infty}),
L^\sigma((ba^{(k)})^\infty)
\big]
$
of $\KK(1)$ with $k\geq1$ the second generation of gaps $(G_{k,n})_{n\geq1}$ in the interval $I_k$ is the family of intervals defined by
$$
G_{k,n}:=
\big(
L^\sigma((a^{(k)}b^{(n+1)})^\infty),
L^\sigma(\overline{b^\infty a^{(k)}b^{(n)}ab^\infty})
\big)
\textrm{ for }
n\geq1.
$$
According to Lemma \ref{LemSecondGenerationCantor}, and assuming that $k=2i$, we have
\begin{eqnarray*}
&&
G_{k}^{(+)}=
L^\sigma
\big(
b^\infty a^{(k+1)}b^\infty
\big)
=
14\cdot[1,4,a^{(i)},b^\infty]
\\
&&
G_{k+1}^{(-)}=
L^\sigma
\big(
(ba^{(k+1)})^\infty
\big)
=
14\cdot[1,4,a^{(i)},(ba^{(k+1)})^\infty],
\end{eqnarray*}
therefore $G_{k}^{(+)}<G_{k+1}^{(-)}$ is equivalent to
$
[b^\infty]<[(ba^{(k+1)})^\infty]
$,
which is true according to Lemma \ref{LemLexicographicOrder}. From the same Lemma it is obvious that for $k=2i$ and $i\to+\infty$ we have
$$
G_{k}^{(+)}=
14\cdot[1,4,a^{(i)},b^\infty]
\to
14\cdot[1,4,a^\infty]
=
\phi_\infty.
$$
The same estimate holds for $k=2i+1$. A similar argument proves the analogous statement for the second generation of holes $(G_{k,n})_{n\geq1}$, where for any $k\geq1$ and $n\geq1$ the formula for the endpoints $G_{k,n}{(-)}$ and $G_{k,n}^{(+)}$ is given by Lemma \ref{LemSecondGenerationCantor}. Theorem \ref{TheoremMainTheorem} is proved.

\subsection*{Acknowledgements}
C. Ulcigrai is currently supported by the ERC Grant ChaParDyn. The authors  would  like to thank the hospitality given by the Max Planck Institute during the program \emph{Dynamics and Numbers} and by CIRM in Luminy, Marseille, during the conference \emph{Dynamics and Geometry in the Teichmueller Space},  where parts of this project were completed. Research visits which made this collaboration possible were also supported by the EPSRC Grant EP/I019030/1, and by the ERC Grant ChaParDyn. The research leading to these results has received funding from the European Research Council under the European Union's Seventh Framework Programme (FP/2007-2013) / ERC Grant Agreement n. 335989.

\end{document}